\documentclass[xcolor=dvipsnames,svgnames,table,reqno]{amsart}

\input xy
\xyoption{all}
\usepackage{accents}
\usepackage{epsfig}
\usepackage{xcolor}
\usepackage{amsthm}
\usepackage{amssymb}
\usepackage{bbm}
\usepackage{amsmath}
\usepackage{amscd}
\usepackage{amsopn}
\usepackage{graphicx}
\usepackage{xspace}

\usepackage{hhline}
\usepackage{easybmat}
\usepackage{caption}   
\usepackage{relsize}

\usepackage{url}
\usepackage{enumitem, hyperref}\hypersetup{colorlinks}

\usepackage{stmaryrd}

\usepackage[T1]{fontenc}

\definecolor{mor}{RGB}{155,5,255}

\colorlet{purpleB70}{blue!70!red}

\colorlet{orangeR65}{red!65!yellow}

\definecolor{red2}{HTML}{d41173}

\definecolor{neongreen}{HTML}{1bf702}

\definecolor{radicalred}{HTML}{FF355E}

\definecolor{denim}{HTML}{1560BD}

\definecolor{darkcyan}{rgb}{0.0, 0.55, 0.55}

\definecolor{cilek}{HTML}{FF43A4}

\definecolor{lavender}{HTML}{9F00C5}

\definecolor{phlox}{rgb}{0.87, 0.0, 1.0}

\definecolor{fluorescentpink}{HTML}{FF1493}

\definecolor{napiergreen}{rgb}{0.16, 0.5, 0.0}

\definecolor{kellygreen}{rgb}{0.3, 0.73, 0.09}

\definecolor{parisgreen}{HTML}{ 50C878 }

\definecolor{palatinateblue}{rgb}{0.15, 0.23, 0.89}

\definecolor{ceruleanblue}{rgb}{0.16, 0.32, 0.75}

\definecolor{brandeisblue}{rgb}{0.0, 0.44, 1.0}

\definecolor{KLMblue}{HTML}{0FC0FC}

\definecolor{cinnamon}{rgb}{0.82, 0.41, 0.12}

\definecolor{darkorange}{rgb}{1.0, 0.55, 0.0}

\definecolor{darktangerine}{rgb}{1.0, 0.66, 0.07}

\definecolor{deepcarrotorange}{rgb}{0.91, 0.41, 0.17}

\definecolor{internationalorange}{HTML}{FF4F00}

\definecolor{persimmon}{HTML}{EC5800}

\definecolor{pumpkin}{HTML}{FF7518}

\definecolor{darkred}{rgb}{1,0,0} 
\definecolor{darkgreen}{rgb}{0,0.7,0}
\definecolor{darkblue}{rgb}{0,0,1}

\hypersetup{colorlinks,
linkcolor=darkblue,
filecolor=darkgreen,
urlcolor=darkred,
citecolor=darkgreen}

\makeatletter
\def\reflb#1#2{\begingroup
    #2%
    \def\@currentlabel{#2}%
    \phantomsection\label{#1}\endgroup
}
\makeatother

\numberwithin{equation}{section}
\newtheorem{Theorem}{Theorem}
\numberwithin{Theorem}{section}

\newtheorem{TheoremX}{Theorem}

\newtheorem{CorollaryX}{Corollary}

\newtheorem   {Lemma}[Theorem]{Lemma}

\newtheorem   {Proposition}[Theorem]{Proposition}
\newtheorem   {Corollary}[Theorem]{Corollary}
\theoremstyle {definition}
\newtheorem   {Definition}[Theorem]{Definition}
\theoremstyle {remark}
\newtheorem   {Remark}[Theorem]{Remark}
\newtheorem   {Example}[Theorem]{Example}

\def    \eps    {\epsilon}

\newcommand{\CA}{{\mathcal A}}
\newcommand{\CC}{{\mathcal C}}

\newcommand{\CalD}{{\mathcal D}}

\newcommand{\CS}{{\mathcal S}}

\newcommand{\supp}{\operatorname{supp}}

\newcommand{\id}{{\mathit id}}

\newcommand{\const}{{\mathit const}}

\newcommand{\fc}{{\mathfrak c}}

\newcommand{\fx}{{\mathfrak x}}

\newcommand{\ty}{\tilde{y}}

\newcommand{\tL}{\tilde{L}}

\newcommand{\tK}{\tilde{K}}

\newcommand{\tU}{\tilde{U}}

\newcommand{\tu}{\tilde{u}}

\newcommand{\tx}{\tilde{x}}

\newcommand{\tpi}{\tilde{\pi}}

\newcommand{\CB}{{\mathcal B}}

\newcommand{\PP}{{\mathcal P}}

\def    \nat    {{\natural}}
\def    \F      {{\mathbb F}}
\def    \AA    {{\mathbb A}}

\def    \R      {{\mathbb R}}

\def    \Z      {{\mathbb Z}}
\def    \N      {{\mathbb N}}

\def    \T      {{\mathbb T}}
\def    \CP     {{\mathbb C}{\mathbb P}}

\def    \12     {{\frac{1}{2}}}

\def    \p      {\partial}
\def    \codim  {\operatorname{codim}}

\def    \im     {\operatorname{im}}

\def    \HF     {\operatorname{HF}}

\def    \H      {\operatorname{H}}

\def    \CF      {\operatorname{CF}}

\def    \vol     {\operatorname{vol}}

\def    \by      {\bar{y}}

\def    \hn    {\scriptscriptstyle{H}}
\def    \Fl    {\scriptscriptstyle{Fl}}

\newcommand    \htop  {\operatorname{h_{\scriptscriptstyle{top}}}}
\newcommand    \hvol  {\operatorname{h_{\scriptscriptstyle{vol}}}}
\newcommand   \hbr {\operatorname{\hbar}}

\newcommand \Span   {\operatorname{span}}

\begin{document}


\setlength{\smallskipamount}{6pt}
\setlength{\medskipamount}{10pt}
\setlength{\bigskipamount}{16pt}





\title [Topological Entropy and Floer Theory]{Topological Entropy of
  Hamiltonian Diffeomorphisms: a Persistence Homology and Floer Theory
  Perspective}

\author[Erman \c C\. inel\. i]{Erman \c C\. inel\. i}
\author[Viktor Ginzburg]{Viktor L. Ginzburg}
\author[Ba\c sak G\"urel]{Ba\c sak Z. G\"urel}

\address{E\c C: ETH Z\"urich, R\"amistrasse 101, 8092 Z\"urich,
  Switzerland} \email{erman.cineli@math.eth.ch}

\address{VG: Department of Mathematics, UC Santa Cruz, Santa Cruz, CA
  95064, USA} \email{ginzburg@ucsc.edu}

\address{BG: Department of Mathematics,
  UCF Orlando, Orlando, FL 32816, USA} \email{basak.gurel@ucf.edu}



\date{\today} 

\thanks{The work is partially supported by NSF CAREER award
  DMS-1454342 (BG), Simons Foundation Collaboration Grants 581382 (VG)
  and 855299 (BG) and ERC Starting Grant 851701 via a postdoctoral
  fellowship (E\c{C})}


\begin{abstract}
  We study topological entropy of compactly supported Hamiltonian
  diffeomorphisms from a perspective of persistent homology and
  Floer theory. We introduce barcode entropy, a Floer-theoretic
  invariant of a Hamiltonian diffeomorphism, measuring exponential
  growth under iterations of the number of not-too-short bars in the
  barcode of the Floer complex. We prove that the barcode entropy is
  bounded from above by the topological entropy and, conversely, that
  the barcode entropy is bounded from below by the topological entropy
  of any hyperbolic invariant set, e.g., a hyperbolic horseshoe. As a
  consequence, we conclude that for Hamiltonian diffeomorphisms of
  surfaces the barcode entropy is equal to the topological entropy.
\end{abstract}

\maketitle

\vspace{-0.2in}

\tableofcontents

\section{Introduction}
\label{sec:intro}
In this work we study topological entropy of compactly supported
Hamiltonian diffeomorphisms from a perspective of Floer theory, using
the machinery of persistent homology.  We introduce a Floer-theoretic
invariant of a Hamiltonian diffeomorphism, which we call barcode
entropy, measuring roughly speaking the rate of exponential growth
under iterations of the number of bars (of length greater than
$\eps>0$) in the barcode of the Floer complex. This invariant comes in
two forms: the absolute barcode entropy associated with the
Hamiltonian Floer complexes of the iterates of the diffeomorphism and
the relative barcode entropy arising from the Lagrangian Floer complex
of a fixed Lagrangian submanifold and its iterated images.

Barcode entropy can be thought of as a Floer theory counterpart of
topological entropy and the two invariants are closely related.  We
show that the barcode entropy (absolute and relative) is bounded from
above by the topological entropy (Theorem \ref{thm:A} and Corollary
\ref{cor:A}) and, conversely, that the absolute barcode entropy is
bounded from below by the topological entropy of any (uniformly)
hyperbolic invariant set, e.g., a horseshoe; see Theorem
\ref{thm:B}. In particular, as a consequence of these two bounds and a
work of Katok, \cite{Ka80}, we conclude in Theorem~\ref{thm:C} that
the absolute barcode entropy is equal to the topological entropy for
Hamiltonian diffeomorphisms of closed surfaces.

The crux of the paper lies in the definition of barcode entropy and
its connection to topological entropy.  The very existence of an
invariant relating features of the Hamiltonian or Lagrangian Floer
complex to topological entropy is not obvious. The only prior
indication known to us that this is indeed possible in dimension two
for relative barcode entropy comes from the main theorem in \cite{Kh}
and also \cite[Prop.\ 3.1.10]{Hu} discussed in more detail in Remark
\ref{rmk:thm-A-2D}.

To detect topological entropy, one has to extract from the Floer
complexes of the iterates an amount of information, up to an
$\eps$-error, growing exponentially with the order of iteration.
\emph{A priori} it is unclear if the complexes carry this much
information and, if so, how to extract it. (The $\eps$-error clause is
essential.) For instance, it is not obvious how and if a high-entropy
horseshoe localized to a small ball would register on the level of
Floer complexes. One apparent difficulty is that in most cases the
effective diameter of the action spectrum grows sub-exponentially with
the order of iteration; see Remark \ref{rmk:spec}. One can think of
Floer theory as a filter or an intermediate device between a dynamical
system and the observer, and it is not clear if it lets through enough
information to detect topological entropy in the Hamiltonian
setting. (See also Remark \ref{rmk:po-entropy} for a different
perspective.)

To the best of our knowledge, there are only two other settings where
connections between topological entropy--type invariants and
symplectic topology have been studied. The first setting concerns
topological entropy of Reeb flows and the growth of various flavors of
contact homology and, in a similar vein, the second one deals with
topological entropy of symplectomorphisms (or contactomorphisms) and
again the growth of Floer homology. (We do not touch upon slow
entropy, for this is ultimately an invariant of a very different
nature; see, however, Remark \ref{rmk:slow}.) This is an extensively
studied subject and we will elaborate on the results in Remark
\ref{rmk:Reeb}. Here we only mention that the underlying theme is that
positivity of topological entropy is obtained as a consequence of
exponential growth of some variant of Floer or contact homology. By
contrast, for Hamiltonian diffeomorphisms, the Floer homology is
independent of the order of iteration and there is no homology growth.
In fact, in the relative case, the Floer homology can even be zero. In
our setting, topological entropy is related to a non-robust (i.e.,
depending on the map) invariant of the Floer complex.

Barcode entropy, the key notion introduced in the paper, relies in a
crucial way on the language of persistent homology. Following
\cite{PS}, this machinery has become one of the standard tools in
studying the dynamics of Hamiltonian diffeomorphisms by symplectic
topological methods and, more generally, in symplectic dynamics,
although the class of problems it has been applied to is quite
different from the exponential growth questions we focus on here. The
barcode of a Floer complex encompasses completely robust invariants of
the system, such as spectral invariants and Floer homology, and also
more fragile features via finite bars. We refer the reader to
\cite{PRSZ, UZ} for a general introduction and to \cite{ADMM, BHS,
  CGG, GU, KS, LRSV, StZh, Sh:HZ, Su} for an admittedly incomplete
collection of sample results. One key point in some of these works and
also in, e.g., \cite{GG:gaps, GG:nc, Gu:nc, Or, Us0, Us}, not using
barcodes directly, is that finite bars, i.e., relatively fragile
features of a barcode, carry information related to interesting
dynamical properties of the system. The present paper builds on this
point.

The paper is organized as follows. In Section \ref{sec:intro2} we
state and extensively discuss the main definitions and results. In
Section \ref{sec:prelim} we set our conventions and notation and
briefly review relevant facts about filtered Lagrangian and
Hamiltonian Floer homology and barcodes following mainly
\cite{Us,UZ}. We return to the definition of relative barcode entropy
in Section \ref{sec:def+prop}, where we state its minor generalization
and also some of its properties. In Section \ref{sec:pf-A} we
generalize and prove Theorem \ref{thm:A}. Finally, Theorems
\ref{thm:B} and \ref{thm:C} are proved in Section \ref{sec:pf-BC},
where we also touch upon an \emph{a priori} lower bound on the
$\gamma$-norm of the iterates in the presence of a hyperbolic set; see
Proposition \ref{prop:gamma}.

\medskip\noindent{\bf Acknowledgements.} We are grateful to Anton
Gorodetski, Boris Hasselblatt and Yakov Pesin for illuminating
discussions and explanations pertaining to the proof of Theorem
\ref{thm:C} and to Sylvain Crovisier for pointing out that the local
maximality condition was not necessary in the first version of Theorem
\ref{thm:B}. Our special thanks are due to Otto van Koert for
stimulating discussions and a computer calculation in 2017 of the
indices and actions of periodic orbits in Smale's horseshoe. We would
also like to thank Marcelo Alves, Leonid Polterovich, Felix Schlenk
and Michael Usher for useful remarks and suggestions.

\section{Key definitions and results}
\label{sec:intro2} 
\subsection{Definitions}
\label{sec:def}
The key new notion introduced in the paper -- barcode entropy -- comes
in two versions: absolute and relative.

Let us start by briefly describing the setting in which these variants
of barcode entropy are defined.  Consider a symplectic manifold $M$
and a closed monotone Lagrangian submanifold $L\subset M$ and a second
Lagrangian submanifold $L'$ Hamiltonian isotopic to $L$. Denote by
$\Lambda$ the universal Novikov field over the ground field
$\F=\F_2$. Informally, $\Lambda$ can be thought of as the field of
Laurent series with coefficients in $\F$ with real (rather than
integer) exponents. The ambient manifold $M$ is not required to be
closed, but it has to have a sufficiently nice structure at infinity,
e.g., to be ``tame'' or convex. In addition, the minimal Maslov number
$N_L$ of $L$ needs to be at least 2. We refer the reader to Section
\ref{sec:prelim} for our conventions and notation, precise definitions
and further details.

Assuming first that $L$ and $L'$ are transverse, we have the filtered
Floer complex $\CF(L,L')$ which is a finite-dimensional vector space
over $\Lambda$ generated by $L\cap L'$. (The grading of the Floer
complex and homology is immaterial for our purposes.)  Denote by
$\CB(L,L')$ the barcode of $\CF(L,L')$ over $\Lambda$; see Section
\ref{sec:barcodes}. Note that in this case (i.e., whenever $L$, $L'$
are transverse) the barcode $\CB(L,L')$ is a finite set.  For
$\eps>0$, let $b_\eps(L,L')$ be the number of bars of length greater
than $\eps$ in this barcode:
\begin{equation}
  \label{eq:b-eps}
  b_\eps(L,L'):=\big|\{\textrm{bars of length greater than
    $\eps$ in $\CB(L,L')$}\}\big|.
\end{equation}

This definition extends in a straightforward way, essentially by
continuity, to the case when the manifolds are not necessarily
transverse. For instance, when $\eps$ is outside the closure
$\bar{\CalD}(L,L')$ of the action difference set (see Section
\ref{sec:spectrum}) we can set
$$
b_\eps(L,L')=b_\eps(L,\tL'),
$$
where $\tL'\pitchfork L$ and $\tL'$ is $C^\infty$-close and
Hamiltonian isotopic to $L'$; see \eqref{eq:b-eps2}. We refer the
reader to Section \ref{sec:def2} and, in particular, \eqref{eq:b-eps3}
for the definition in the general case, and to \cite[Chap.\ 6]{PRSZ}
for other appearances of $b_\eps$. Here we note that in the
non-transverse case the barcode $\CB(L,L')$ may contain infinitely
many bars. However, $b_\eps(L,L') < \infty$ for all $\eps>0$.

Let $\varphi=\varphi_H\colon M\to M$ be a compactly supported
Hamiltonian diffeomorphism. Set $L^k=\varphi^k(L)$.

\begin{Definition}[Relative Barcode Entropy, I]
  \label{def:hbr-rel1}
  The \emph{$\eps$-barcode entropy of $\varphi$ relative to $L$} is
  $$
  \hbr_\eps(\varphi;L):=\limsup_{k\to \infty}\frac{\log^+
    b_\eps\big(L,L^k\big)}{k}
  $$
  and the \emph{barcode entropy of $\varphi$ relative to $L$} is
  $$
  \hbr(\varphi;L):=\lim_{\eps\searrow 0} \hbr_\eps(\varphi, L) \in
  [0,\,\infty].
  $$
\end{Definition}
Here and throughout the paper the logarithm is taken base 2 and
$\log^+:=\max\{\log,0\}$.  Observe that $\hbr_\eps(\varphi, L)$ is
increasing as $\eps\searrow 0$, and hence the limit in the definition
of $\hbr(\varphi,L)$ exists although \emph{a priori} it can be
infinite.  We also emphasize again that the Floer homology $\HF(L)$ is
immaterial for this construction beyond the fact that it is
defined. For instance, $L$ can be a small circle in a surface with
$\HF(L)=0$.  We will extend this definition to pairs of Lagrangian
submanifolds in Section \ref{sec:def2}.

\begin{Remark}
  \label{rmk:intersections}
  Definition \ref{def:hbr-rel1} might feel somewhat
  counterintuitive. The underlying idea is that the barcode counting
  function $b_\eps(L, L')$ gives a lower bound on the number of
  transverse intersections, which is in some sense stable under small
  perturbations with respect to the Lagrangian Hofer distanse
  $d_{\hn}$. For instance, assume that Lagrangian submanifolds $L$,
  $L'$ and $L''$ are Hamiltonian isotopic, $L''\pitchfork L$ and
  $d_{\hn}(L',L'')<\delta/2$. Then, regardless of whether $L$ and $L'$
  are transverse or not, we have
  $$
  |L\cap L''|\geq b_{\eps+\delta}(L,L')
  $$
  for any $\eps\geq 0$; see Sections \ref{sec:barcodes} and
  \ref{sec:def2}.  This would not be true if we replaced
  $b_{\eps+\delta}(L,L')$ by $|L\cap L'|$: nearby intersections can be
  eliminated by a $C^\infty$-small perturbation.
\end{Remark}
 
Let now $M$ be a closed monotone symplectic manifold and again let
$\varphi=\varphi_H\colon M\to M$ be a Hamiltonian diffeomorphism. Then
we can apply the above constructions to $L=\Delta$, the diagonal in
the symplectic square $\big(M\times M, -\omega \oplus \omega\big)$,
with $\varphi$ replaced by $\id\times \varphi$, or directly to the
Floer complex $\CF(\varphi)$ of $\varphi$ \emph{for all free homotopy
  classes of loops in $M$}. For instance, denoting by $\CB(\varphi)$
the barcode of $\CF(\varphi)$ over $\Lambda$, we have
\begin{align*}
  b_\eps\big(\varphi^k\big)
  &=\big|\{\textrm{bars of length greater than
    $\eps$ in the barcode $\CB\big(\varphi^k\big)$}\}\big|\\
  &=b_\eps\big(L,L^k\big),
\end{align*}
where $L=\Delta$ and $L^k$ is the graph of $\varphi^k$.

\begin{Definition}[Absolute Barcode Entropy]
  \label{def:hbr-bar}
  The \emph{$\eps$-barcode entropy} of $\varphi$ is
  $$
  \hbr_\eps(\varphi):=\limsup_{k\to \infty}\frac{\log^+
    b_\eps\big(\varphi^k\big)}{k}
  $$
  and the \emph{(absolute) barcode entropy} of $\varphi$ is
  $$
  \hbr(\varphi):=\lim_{\eps\searrow 0} \hbr_\eps(\varphi) \in
  [0,\,\infty]
  $$
  or, in other words,
  $$
  \hbr(\varphi):=\hbr(\id\times\varphi;\Delta).
  $$
 \end{Definition}
 Here again $\hbr_\eps(\varphi)$ is increasing as $\eps\searrow 0$,
 and hence the limit in the definition of $\hbr(\varphi)$ exists. Note
 that in this definition, in contrast with the relative barcode
 entropy, we can work with any ground field $\F$ as long as $M$ is
 monotone.

 In this paper we are primarily interested in absolute barcode entropy
 while relative entropy plays a purely technical role, arising
 naturally in our approach to the proof of Corollary \ref{cor:A} via
 Theorem \ref{thm:A}. We will revisit the definitions and briefly
 touch upon general properties of barcode entropy in Section
 \ref{sec:def+prop}.

\subsection{Main results}
\label{sec:results}
With the definition of barcode entropy in place, we are ready to state
the main results of the paper, which ultimately justify the
definition.
\begin{TheoremX}
  \label{thm:A}
  Let $L$ be a closed monotone Lagrangian submanifold with minimal
  Chern number $N_L\geq 2$ in a symplectic manifold $M$ and let
  $\varphi\colon M\to M$ be a compactly supported Hamiltonian
  $C^\infty$-diffeomorphism. Then
   $$
   \hbr(\varphi; L)\leq \htop(\varphi).
   $$
 \end{TheoremX}
 Note that since $\varphi$ is compactly supported, the lack of
 compactness of $M$, provided that it is ``tame'' at infinity, causes
 no additional problems. We can set
 $\htop(\varphi):=\htop(\varphi|_{\supp(\varphi)})$ or equivalently
 $\htop(\varphi):=\htop(\varphi|_{X})$ for any compact set
 $X\supset \supp(\varphi)$.  Variants of Theorem \ref{thm:A} also hold
 in some other cases; see, e.g., Remark \ref{rmk:filtr-other3}.

 Since $\htop(\id\times\varphi)=\htop(\varphi)$, as an immediate
 consequence of Theorem \ref{thm:A}, we have the following.
 
 \begin{CorollaryX}
   \label{cor:A}
   Let $\varphi\colon M\to M$ be a Hamiltonian
   $C^\infty$-diffeomorphism of a closed monotone symplectic manifold
   $M$. Then
   $$
   \hbr(\varphi)\leq \htop(\varphi).
   $$
\end{CorollaryX}

Another interesting consequence of Theorem \ref{thm:A}, not obvious
from the definitions, is that $\hbr(\varphi; L)<\infty$ and, in
particular, $\hbr(\varphi)<\infty$. In contrast with Theorem
  \ref{thm:A}, here and in Theorems \ref{thm:B} and \ref{thm:C} the
ground field $\F$ can have any characteristic.

\begin{Remark}[Growth of periodic points]
  \label{rmk:PO-Growth}
  One cannot replace the number of bars $b_\eps(\varphi^k)$ or
  $b_\eps(L,L^k)$ in the definition of barcode entropy by the total
  number of $k$-periodic points or Lagrangian intersections, while
  keeping Theorem \ref{thm:A} and Corollary \ref{cor:A}. Indeed, in
  dimension two, the number of periodic points can grow arbitrarily
  fast, and moreover super-exponential growth is in some sense
  typical; see \cite{As}. In higher dimensions, a smooth zero-entropy
  map may have super-exponential orbit growth; \cite{Kal}. Hence, in
  both cases the exponential growth rate of the number of periodic
  points could in general be infinite.
\end{Remark}

\begin{Remark}[Idea of the proof of Theorem \ref{thm:A}]
  As in many results of this type (see, e.g., \cite{Me}), Theorem
  \ref{thm:A} is ultimately based on Yomdin's theorem relating
  topological entropy to the rate of exponential volume growth; see
  \cite{Yo} and also \cite{Gr}. (Hence, the requirement that $\varphi$
  is $C^\infty$-smooth is essential.) Let us briefly outline the idea
  of the proof assuming that all intersections are transverse. For a
  small $\eps>0$, the barcode entropy is roughly the rate of
  exponential growth of $b_\eps\big(L,L^k\big)$ as $k\to\infty$. By
  Remark \ref{rmk:intersections}, this rate bounds from below the rate
  of exponential growth of
  $N_k(\tilde{L}):=\big|\tilde{L}\cap L^k\big|$ for any Lagrangian
  submanifold $\tilde{L}$ Hamiltonian isotopic and $d_{\hn}$-close to
  $L$ with the upper bound on $d_{\hn}(L,\tilde{L})$ completely
  determined by $\eps$.  We construct a \emph{Lagrangian tomograph}: a
  family of such Lagrangian submanifolds $\tilde{L}=L_s$, independent
  of $k$ and parametrized by $s$ in some ball $B$, so that
  $$
  \int_B N_k(L_s)\, ds\leq \const \cdot\vol \big(L^k\big)
  $$
  by a variant of Crofton's inequality; cf.\ \cite{Ar1,Ar2} where a
  similar construction is used. Now Yomdin's theorem gives a lower
  bound on $\htop(\varphi)$. Note that in contrast with some other
  arguments of this type (see, e.g., \cite{Al2,FS:slow} and references
  therein), the ball $B$ can possibly have very large dimension and
  the map $\Psi\colon B\times L \to M$ sending $(s,x)$ to the image of
  $x$ on $L_s$ need not be a fibration but only a submersion onto its
  image.
\end{Remark}

\begin{Remark}
  Instead of working with the class of monotone Lagrangian
  submanifolds $L$ one can require $L$ to be oriented, relatively spin
  and weakly unobstructed after bulk deformation as in \cite{FOOO} and
  replaced the coefficient field $\F_2$ by a field of zero
  characteristic. We expect Theorem \ref{thm:A} to still hold in this
  setting and the proof to go through word-for-word.
\end{Remark}

We do not view Theorem \ref{thm:A} or Corollary \ref{cor:A} as an
effective method to calculate $\htop(\varphi)$, but rather as a result
connecting two notions of entropy lying in completely disparate
domains. Yet, this result would be meaningless and hold trivially if
$\hbr(\varphi,L)$ were always zero. (As an example,
$\hbar(\varphi_H)=0$ when $H$ is an autonomous Hamiltonian on a
surface. This fact is a consequence of Theorem \ref{thm:C} below and
is not directly obvious. An alternative approach would be to show that
in this case the volume of the graph of $\varphi^k$ grows
sub-exponentially and then invoke the proof of Theorem \ref{thm:A} or,
when $H$ is real analytic or Morse, one can verify directly that
$b_\eps\big(\varphi^k\big)$ grows at most polynomially.)

Here we focus on absolute barcode entropy, and the next two theorems
show that often $\hbar(\varphi)\neq 0$ and, in fact, the two notions
of entropy are perhaps (numerically) closer to each other than one
might expect.

In the next theorem we are concerned with (uniformly)
\emph{hyperbolic} invariant sets. Recall that an invariant set
$K\subset M$ is said to be hyperbolic if for some Riemannian metric on
$M$ there exist positive constants $\lambda_- < 1 < \lambda_+$ and a
splitting $T_xM=E^-_x\oplus E^+_x$ for every $x\in K$, invariant under
$D\varphi$, such that
$$
\|D\varphi|_{E^-_x}\|\leq \lambda_-\textrm{ and }
\|D\varphi|_{E^+_x}\|\geq \lambda_+
$$
for all $x\in K$. We refer the reader to, e.g., \cite[Sect.\ 6]{KH}
for a detailed discussion of hyperbolicity. Here such sets are
required to be compact by definition.

\begin{TheoremX}
   \label{thm:B}
   Let $\varphi\colon M\to M$ be a Hamiltonian diffeomorphism of a
   closed monotone symplectic manifold $M$ and let $K\subset M$ be a
   hyperbolic invariant subset. Then

   $$
   \hbr(\varphi)\geq \htop(\varphi|_K).
   $$
\end{TheoremX}

The key to the proof of this theorem is the fact that a Floer
trajectory $u$ asymptotic to a periodic orbit $x$ of any period with
$x(0)\in K$ must have energy bounded away from zero by a certain
constant $\eps_K>0$ independent of $u$; see Proposition
\ref{prop:energy}. The proof of the proposition is based on the
Crossing Energy Theorem, \cite[Thm.\ 6.1]{GG:PR}, and the Anosov
Closing Lemma, \cite[Thm.\ 6.4.15]{KH}. In contrast with other results
of this paper, it would be sufficient in Theorem \ref{thm:B} to assume
that $\varphi$ is only $C^2$-smooth.

Among the examples of sets $K$ meeting the requirements of Theorem
\ref{thm:B} are hyperbolic horseshoes and then $\htop(\varphi|_K)>0$;
see the discussion in Section \ref{sec:pf-C}. Hence, whenever
$\varphi$ has such a horseshoe, which is a common occurrence,
$\hbr(\varphi)>0$.

In dimension two, $\htop(\varphi)$ is the supremum of
$\htop(\varphi|_K)$ over all $K$ as in Theorem \ref{thm:B}; see
\cite{Ka80} and also \cite[Suppl.\ S by Katok and Mendoza]{KH}. Hence,
by Theorem \ref{thm:A} (or Corollary \ref{cor:A}) and Theorem
\ref{thm:B}, we have the following.

 \begin{TheoremX}
   \label{thm:C}
   Let $\varphi\colon M\to M$ be a Hamiltonian
   $C^\infty$-diffeomorphism of a closed surface $M$. Then
   \begin{equation}
     \label{eq:bar-top}
   \hbr(\varphi)= \htop(\varphi).
 \end{equation}
\end{TheoremX}
We will prove and further discuss Theorem \ref{thm:C} in
Section~\ref{sec:pf-BC}.

A bonus consequence of the proof of Theorem \ref{thm:B} is that the
$\gamma$-norm, and hence the Hofer norm, of the iterates $\varphi^k$
is bounded away from zero when, for instance, $\varphi$ has a
horseshoe or sufficiently many hyperbolic fixed points and also in
some other cases; see Section \ref{sec:gamma} and
Proposition~\ref{prop:gamma}.

Another consequence of the proof of Theorem \ref{thm:A} and Theorem
\ref{thm:C} is a relation between the barcode and topological entropy
of $\varphi$ and the exponential growth rate of the volume of the
graph of $\varphi^k$; see Section \ref{sec:vol}.

\begin{Remark}    
  Since this work appeared as a preprint, a variant of Theorem
    \ref{thm:B} for Lagrangian intersections has been proved in
    \cite{Me:Entropy}.
\end{Remark}  

\subsection{Discussion and further remarks}
\label{sec:discussion}
Our definition of barcode entropy is quite general and can be applied
to any sequence of persistence modules. (See, e.g., \cite{Ca} and also
\cite{PRSZ} for an introduction to persistent homology theory and
further references.) Then barcode entropy can be thought of as the
rate of exponential growth of information carried by this sequence,
provided that we take $b_\eps$ as a measure of the amount of
information. With this in mind, we expect the notion to be useful in
some other settings. However, from a purely algebraic perspective the
definition is also quite primitive: in contrast with topological
entropy, one cannot infer any interesting properties of barcode
entropy directly from the definition; cf.\ Section \ref{sec:prop}.

Yet, Theorems \ref{thm:A}, \ref{thm:B} and \ref{thm:C} show that this
is the ``right'' definition in the context of Floer theory and that in
this setting barcode entropy has numerous properties not formally
following from the definition. The reason probably is that Floer
complexes have rich additional features, and it would be interesting
to understand what properties of Floer complexes ``make this
definition work'' -- the proofs do not answer this question.

In the rest of this section we will further comment on our main
results.

\begin{Remark}[Theorem \ref{thm:A} when $\dim M=2$]
  \label{rmk:thm-A-2D}
  The assertion of Theorem \ref{thm:A} is already nontrivial when $L$
  is an embedded loop (a Lagrangian submanifold) on a surface $M$,
  even when $L=S^1$ is the zero section in $M= T^*S^1=S^1\times \R$ or
  $L$ is a meridian or a parallel in $M=S^1\times S^1$. However, in
  this situation, one might expect to have a more general criterion.
  Namely, the results from \cite{Kh} building on \cite{BM} indicate
  that when $\dim M=2$, for a broad class of loops $L$, there could be
  a sufficient condition, aka a criterion, for $\htop(\varphi)>0$
  expressed solely in terms of $\CF\big(L,\varphi(L)\big)$ without
  using the iterates $\varphi^k$, similarly to the forced entropy
  results. (See, e.g., \cite{AP} and references therein.) However, no
  such a criterion is known; nor if this is a purely low-dimensional
  phenomenon. Furthermore, for a fixed $L$, such an entropy positivity
  condition, even if it exists, could not be necessary, i.e., it
  cannot possibly always detect positive entropy; see Example
  \ref{exm:moving-L} below. Also note that \cite[Prop.\ 3.1.10]{Hu},
  based on \cite{Kh:Th}, gives a lower bound for the topological
  entropy of a Hamiltonian diffeomorphism $\varphi\colon S^2\to S^2$
  in terms of the linear growth rate of the Hofer distance between $L$
  and $\varphi^k(L)$, where $L$ is the equator. (To the best of our
  knowledge, there are no examples where this distance is shown to
  grow linearly, although conjecturally such examples exist.)
\end{Remark}

\begin{Example}[Strict inequality in Theorem \ref{thm:A}]
  \label{exm:moving-L}
  In contrast with Theorem \ref{thm:C}, in the setting of Theorem
  \ref{thm:A} the inequality can be strict even when $\dim M=2$ (and
  hence $\dim L=1$). Indeed, assume that $L$ is contained in a region
  where $\varphi=\id$, but $\varphi$ has a hyperbolic horseshoe
  elsewhere. Then $\hbr(\varphi;L)=0$, but $\htop(\varphi)>0$. Also
  note that in dimension two $\hbr(\varphi;L)\leq \hbr(\varphi)$ as a
  consequence of Theorem \ref{thm:C}, and it would be very
  illuminating to understand if this is true in general.
\end{Example}

\begin{Remark}[Non-compact version of Corollary \ref{cor:A}]
  The corollary readily extends to compactly supported Hamiltonian
  diffeomorphisms $\varphi=\varphi_H$ of symplectic manifolds $M$
  sufficiently ``tame'' at infinity. For instance, $M$ can be convex
  at infinity or wide in the sense of \cite{Gu}. To define the barcode
  entropy of $\varphi$, we fix a proper autonomous Hamiltonian $Q$
  without 1-periodic orbits at infinity and vanishing on $\supp
  H$. (For instance, $Q$ can be a small positive definite quadratic
  form outside a large ball when $M=\R^{2n}$.) Next, let $\psi$ be a
  small non-degenerate perturbation of $\varphi_Q\varphi^k$ coinciding
  with $\varphi_Q$ at infinity. The filtered Floer complex of $\psi$
  is defined and can be used in place of $\CF\big(\varphi^k\big)$ to
  define the barcode entropy $\hbar(\varphi)$ of $\varphi$. (This
  complex depends on $Q$ and the perturbation, but $b_\eps(\psi)$ is
  independent of the perturbation, although it might still depend on
  $Q$.) Now the proof of Corollary \ref{cor:A} via Theorem
  \ref{thm:A2} goes through, establishing the inequality
  $\hbar(\varphi)\leq \htop(\varphi)$, when we let $L_0$ be the
  diagonal in $M\times M$ and replace $L^k$ by the graph of $\psi$;
  see Section \ref{sec:pf-A2}. In a similar vein, Theorems \ref{thm:B}
  and \ref{thm:C} also hold with the definition of barcode entropy
  suitably adjusted when $M$ is ``open'' but $\varphi$ is compactly
  supported.
\end{Remark}

\begin{Remark}[Free homotopy classes of loops, I]
  \label{rmk:hom-classes}
  The requirement that we work with the Floer complex for the entire
  set $\tpi_1(M)$ of free homotopy classes of loops in $M$ is
  absolutely essential and without it the theorems would not hold; see
  Example \ref{ex:hom-classes}. In the Lagrangian setting, its
  counterpart is commonly assumed: the Lagrangian Floer complex
  $\CF(L_0,L_1)$ is usually defined for all free homotopy classes of
  paths from $L_0$ and $L_1$. On the other hand, the Hamiltonian Floer
  complex $\CF(\varphi)$ is traditionally defined only for
  contractible loops. This is not the case in Definition
  \ref{def:hbr-bar} where we use the Floer complex associated with the
  entire collection $\tpi_1(M)$ of free homotopy classes of loops in
  $M$. In fact, one can associate a variant of barcode entropy
  $\hbr_X$ to any subset $X\subset \tpi_1(M)$, and clearly,
  $\hbr_Y\leq \hbr_X$ when $Y\subset X$.  Then Theorem \ref{thm:A}
  would still obviously hold, for $\hbar_X\leq \hbar$. However,
  Theorems \ref{thm:B} and \ref{thm:C} would in general fail when
  $X\neq \tpi_1(M)$; see Example \ref{ex:hom-classes}. This phenomenon
  is unrelated to algebraic topological conditions for entropy
  positivity (e.g., the Entropy Conjecture or the fundamental group
  growth; see \cite{Al1,Ka:E50,KH,Yo} and references therein): Theorem
  \ref{thm:C} holds for $M=S^2$ and for $M=\T^2$. We will elaborate on
  this in Remark \ref{rmk:hom-classes2}.
\end{Remark}

\begin{Example}[Failure of Theorems \ref{thm:B} and \ref{thm:C} for
  restricted free homotopy classes of loops]
  \label{ex:hom-classes}
  Let $U$ be a small disk in $\T^2$ and let $K$ be a time-dependent
  Hamiltonian with $\supp K\subset U$ generating a positive-entropy
  time-one map $\varphi_K$. Next, let $F$ be an autonomous Hamiltonian
  on $\T^2$ such that under the flow $\varphi^t_F$ for time
  $t\in [0,1]$ every point in $U$ traces a non-contractible loop in a
  class $\fc\in\tpi_1(\T^2)=\Z^2$ and $\varphi_F^1|_U=\id$.  Define
  $H$ by concatenating the Hamiltonians $F$ and $K$, and set
  $\varphi=\varphi^1_H$. Thus $\varphi|_U=\varphi_K$ and
  $\varphi|_{\T^2\setminus U}=\varphi_F|_{\T^2\setminus U}$. It
  follows that $\htop(\varphi)=\htop(\varphi_K)>0$.  However,
  $\hbar_X(\varphi)=0$ whenever $|X\cap \fc\N|<\infty$. Indeed, pick
  $\fx\in \tpi_1(\T^2)$ and assume that $k$ is so large that
  $\fx\neq k\fc$. Then, since every $k$-periodic orbit of $\varphi_H$
  starting in $U$ is in the class $k\fc$, the maps $\varphi_H$ and
  $\varphi_F$ have the same $k$-periodic orbits in the class $\fx$ and
  the actions are equal. A homotopy from $H$ to $F$ which phases $K$
  out does not affect $k$-periodic orbits in the class $\fx$. As a
  consequence, $\varphi_H^k$ and $\varphi_F^k$ have the same filtered
  Floer homology in this class and the same barcode; see, e.g.,
  \cite[Sec.\ 3.2.3]{Gi:Coiso} and references therein for similar
  arguments.  Therefore, $\hbar_X(\varphi_H)=\hbar_X(\varphi_F)$
  whenever $X\cap \fc\N$ is finite. By Theorem \ref{thm:C},
  $\hbar(\varphi_F)=0$ and hence $\hbar_X(\varphi_F)=0$.
\end{Example}

\begin{Remark}[Free homotopy classes of loops, II]
  \label{rmk:hom-classes2} All $k$-periodic orbits of $\varphi_H^t$
  have length bounded from above by $k\|\nabla H\|_{C^0}$. Therefore,
  for a fixed $k\in\N$ only finitely many free homotopy classes can be
  represented by periodic orbits with period up to $k$.  By the Svar\v
  c--Milnor lemma, \cite{Ef}, these classes are contained in the ball
  of radius $R=O(k)$, with respect to the word length metric, in
  $\tpi_1(M)$ centered at the origin. The number of conjugacy classes
  is known to have exponential growth for many hyperbolic-type groups,
  e.g., for $\pi_1(M)$ when $M$ has negative sectional curvature; see
  \cite{CK, HO} and references therein. (We expect this to also be
  true when $M$ a symplectically hyperbolic manifold in the sense of
  \cite{Ke,Po}, but this is not obvious.) Hence, in this case, the
  number of classes populated by periodic orbits of $\varphi$ with
  period less than $k$ can grow exponentially.  When $\varphi$ is such
  that this is indeed the case, we have $\htop(\varphi)>0$; see
  \cite{Iv,Ji} and also \cite{Al1} and references therein. However, in
  general this is not the source of positive entropy in Theorems
  \ref{thm:B} and \ref{thm:C} as Example \ref{ex:hom-classes} and the
  case of $M=S^2$ show.
\end{Remark}

\begin{Remark}[Theorem \ref{thm:C}, periodic orbits and topological
  entropy]
  \label{rmk:po-entropy}
  As stated, Theorem \ref{thm:C} fails when $\dim M>4$; see \cite{Ci}.
  Our admittedly very optimistic conjecture is that \eqref{eq:bar-top}
  holds in all dimensions $C^\infty$-generically for Hamiltonian
  diffeomorphisms.
\end{Remark}

\begin{Remark}[Slow entropy]
  \label{rmk:slow}
  Theorem \ref{thm:A} and Corollary \ref{cor:A} have analogues for
  polynomial growth, slow entropy-type invariants (see, e.g.,
  \cite{FS:slow, KT, Po}) which can be proved essentially by the same
  method with natural modifications. To be more specific, the
  polynomial growth rate of $b_\eps(L,L^k)$, i.e.,
  $\limsup \log^+ b_\eps(L,L^k)/\log k$, gives a lower bound on the
  polynomial growth rate of the volume of $L^k$ and, likewise, the
  polynomial growth rate of $b_\eps(\varphi^k)$ provides a lower bound
  on the polynomial growth rate of the volume of the graph of
  $\varphi^k$.
\end{Remark}

\begin{Remark}[Lower barcode entropy]
  \label{rmk:lower}
  If we replaced the upper limit by the lower limit in the definition
  of barcode entropy, Theorem \ref{thm:A} and Corollary \ref{cor:A}
  would obviously still hold for the so-defined lower barcode
  entropy. However, the proof of Theorem \ref{thm:B} would break down
  because of the upper limit in \eqref{eq:p-htop} and, as a
  consequence, we would not be able to establish Theorem \ref{thm:C}.
\end{Remark}

Theorems \ref{thm:A}, \ref{thm:B} and \ref{thm:C} lead to several
other questions completely left out in this paper. One of them,
alluded to in Example \ref{exm:moving-L}, is if there is a relation
between absolute and relative entropy. Another is a construction of a
Floer theoretic analogue of dynamical Morse entropy introduced in
\cite{BG}. Also, there are several possible alternatives to our
definition of barcode entropy, some of which are examined in
\cite{CGG:growth}. Furthermore, the definition of barcode entropy,
Corollary \ref{cor:A} and Theorems \ref{thm:B} and \ref{thm:C}
generalize to symplectomorphisms symplectically isotopic to the
identity with the Floer complex defined as in \cite{LO}; see
\cite{Pe}. Finally, it would also be interesting to extend the
definition of relative barcode entropy to a broader class of
Lagrangian submanifolds beyond the case when Lagrangian Floer homology
is defined. Such a generalization may have applications to dynamics.

\begin{Remark}[Topological entropy of Reeb flows and
  symplectomorphisms]
  \label{rmk:Reeb}
  As was mentioned in Section \ref{sec:intro}, connections between
  topological entropy of Reeb flows or symplectomorphisms and
  exponential growth of contact or Floer homology of various flavors
  have been extensively studied.  In the Reeb setting, these
  connections generalize well-known relations between the topological
  entropy of the geodesic flow and the geometry or topology of the
  underlying manifold; see \cite{Ka:E50, KH} for further
  references. This study was originally initiated in \cite{MS} and
  since then the subject has been extensively developed in a series of
  papers by Alves and his collaborators (see, e.g., \cite{AASS, AS,
    Al1, Al2, ACH, AM, AP} and also \cite{Me}). In many of these
  results, positivity of topological entropy follows from exponential
  growth of contact or Legendrian contact homology. For instance, in
  \cite{AM}, the authors construct a contact structure on
  $S^{2n-1\geq 7}$ such that every Reeb flow has positive entropy due
  to exponential growth of wrapped Floer homology, which is a robust
  feature independent of the contact from. This approach is
  reminiscent of the Entropy Conjecture and other results where
  algebraic topological features of a map give a lower bound on the
  topological entropy; see \cite{Ka:E50, KH, Yo} for further
  references. The case of symplectomorphisms or contactomorphisms is
  very similar in spirit and closely related to the Reeb setting, and
  the results usually rely again on exponential growth of a variant of
  Floer homology; see, e.g., \cite{Da1, Da2, FS}. On the other hand,
  for compactly supported Hamiltonian diffeomorphisms, there is no
  Floer homology growth. Topological entropy is instead related to
  barcode entropy, a non-robust invariant of the Floer complex.
\end{Remark}

\begin{Remark}
  \label{rmk:Reeb-barcode}
  Since the first version of this work appeared as a preprint,
  significant progress has been made in extending our results and
  constructions to Reeb flows. In particular, barcode entropy is
  defined and the analogues Theorems \ref{thm:A}, \ref{thm:B} and
  \ref{thm:C} are proved for geodesic flows in \cite{GGM}.  More
  generally, for Reeb flows on the boundary of a Liouville domain, the
  barcode entropy is defined in \cite{FLS}, where a version of
  Corollary \ref{cor:A} is also proved, and Theorems \ref{thm:B} and
  \ref{thm:C} are established in \cite{CGGM:entropy}.  Moreover,
  Theorem \ref{thm:A} for wrapped Floer homology is proved in
  \cite{Fe}.
\end{Remark}

\begin{Remark}[Forcing and connections with Hofer's geometry in
  dimension two] Connections of topological entropy in dimension two
  with Hofer's geometry have been recently explored in
  \cite{AM,CM}. The main theme of the latter work is closely related
  to forcing of topological entropy; see \cite{LCT} and references
  therein. The key result of the former is Hofer lower semicontinuity
  of topological entropy for Hamiltonian diffeomorphisms of closed
  surfaces.
\end{Remark}

\section{Preliminaries}
\label{sec:prelim}

\subsection{Conventions and notation}
\label{sec:conv}
For the reader's convenience, we set here our conventions and notation
and briefly recall some basic definitions; see, e.g., \cite{Us} for
further details and references. The reader may want to consult this
section only as needed.

Throughout the paper we assume that the underlying symplectic manifold
$(M,\omega)$ is either closed or ``tame'' at infinity (e.g., convex)
so that the Gromov compactness theorem holds; see, e.g.,
\cite{McDS}. All Lagrangian submanifolds are assumed to be closed
unless explicitly stated otherwise, and \emph{monotone}, i.e., for
some $\kappa\geq 0$, we have
$\left<\omega, A\right>=\kappa \left<\mu_L,A\right>$ for all
$A\in\pi_2(M,L)$, where $\mu_L\in \H^2(M,L;\Z)$ is the Maslov class.
Then $M$ is also monotone with monotonicity constant $2\kappa$, i.e.,
$\left<\omega, A\right>=2\kappa \left<c_1(TM),A\right>$ for all
$A\in\pi_2(M)$. As in \cite{Us}, we define the \emph{minimal Maslov
  number} of $L$ as the positive generator $N_L$ of the subgroup of
$\Z$ generated by $\left<\mu_L, w\right>$ for all maps
$w\colon \AA\to M$ from the cylinder $\AA:=S^{1}\times [0, 1]$ to $M$
sending the boundary $\p\AA$ to $L$. When this group is trivial, we
set $N_L=\infty$.  In what follows, we require that $N_L\geq 2$ unless
$\kappa=0$.  Note that this definition allows $[\omega]$ to vanish on
$\pi_2(M,L)$ and thus includes weakly exact Lagrangian submanifolds.

Alternatively, as has already been pointed out, one can require $L$ to
be oriented, relatively spin and weakly unobstructed after bulk
deformation as in \cite{FOOO} and replaced the coefficient field
$\F=\F_2$ by a field of zero characteristic. We expect that Theorem
\ref{thm:A} extends to this setting.

A \emph{Hamiltonian diffeomorphism} $\varphi=\varphi_H=\varphi_H^1$ is
the time-one map of the time-dependent flow (i.e., a \emph{Hamiltonian
  isotopy}) $\varphi^t=\varphi_H^t$ of a 1-periodic in time
Hamiltonian $H\colon S^1\times M\to\R$, where $S^1=\R/\Z$.  The
Hamiltonian vector field $X_H$ of $H$ is defined by
$i_{X_H}\omega=-dH$.  All Hamiltonians are assumed to be compactly
supported.

The $k$-th iterate $\varphi^k$ is viewed as the time-$k$ map of
$\varphi^t_H$. The $k$-periodic \emph{points} of $\varphi$ are in
one-to-one correspondence with the $k$-periodic \emph{orbits} of $H$,
i.e., of the time-dependent flow $\varphi_H^t$. A $k$-periodic orbit
$x$ of $H$ is said to be \emph{non-degenerate} if the linearized
return map $D_{x(0)}\varphi^k \colon T_{x(0)}M \to T_{x(0)}M$ has no
eigenvalues equal to one. A Hamiltonian $H$ is \emph{non-degenerate}
if all of its 1-periodic orbits are non-degenerate and \emph{strongly
  non-degenerate} when all of its periodic orbits are non-degenerate.

Recall that the \emph{Hofer norm} of $\varphi$ is defined as
$$
\|\varphi\|_{\hn}=\inf_{H}\int_{S^1}\big(\max_M H_t-\min_M H_t\big)\, dt,
$$
where the infimum is taken over all 1-periodic in time Hamiltonians
$H$ generating $\varphi$, i.e., $\varphi=\varphi_H$; see \cite{Ho,
  LMcD, Po:Hofer, Vi}.  The \emph{Hofer distance} between two
Hamiltonian isotopic Lagrangian submanifolds $L$ and $L'$ is
$$
d_{\hn}(L,L')=\inf\left\{\|\varphi\|_{\hn}\mid \varphi(L)=L'\right\};
$$
see \cite{Ch}.

\subsection{Filtered Lagrangian Floer complex}
\label{sec:Lagr}
In this section we very briefly spell out the definition of the
filtered Lagrangian Floer complex we use in this paper. There are
various settings and levels of generality one could work with here.
For the sake of simplicity, we focus on the case of Hamiltonian
isotopic closed monotone Lagrangian submanifolds. (We will touch upon
other settings in Remark \ref{rmk:filtr-other}.) For such Lagrangian
submanifolds, Floer homology was originally defined in \cite{Oh},
albeit in a somewhat different algebraic setting. In this short
narrative we adopt the framework from \cite{Us} with only minor
modifications. We refer the reader to this work and, of course, to
\cite{FOOO} for a much more detailed treatment and further references.

\subsubsection{Floer complex}
\label{sec:Floer-complex}
Let $M$ be a symplectic manifold which is supposed to be sufficiently
``tame'' at infinity (e.g., compact or convex) to guarantee that
compactness theorems hold, and let $L$ and $L'$ be closed monotone
Lagrangian submanifolds intersecting transversely. We assume that $L$
and $L'$ are Hamiltonian isotopic to each other, i.e., there exists a
Hamiltonian isotopy $\varphi_F^t$, $t\in [0, 1]$, such that
$L'=\varphi_F(L)$. For the time being we will treat the isotopy
$\varphi_F^t$ as a part of the data. Furthermore, we require that
$N_{L}\geq 2$, where $N_{L}=N_{L'}$ is the minimal Chern number as
defined in Section \ref{sec:conv}.

Let $\PP(L,L')$ be the space of smooth paths in $M$ from $L$ to $L'$
and $\pi_1(M;L,L')$ be the set of its connected components. For
instance, the intersection points of $L$ and $L'$ are elements of
$\PP(L,L')$; however, these elements might be in different connected
components of $\PP(L,L')$. Fix a reference path $x_\fc$ in each
$\fc\in \pi_1(M;L,L')$. A \emph{capping} $w$ of a path $x\in\fc$ is a
homotopy of $x$ to $x_\fc$ (with end-points on $L$ and $L'$), taken up
to a certain equivalence relation. Namely, two cappings $w$ and $w'$
are equivalent if and only if the cylinder $v$ obtained by attaching
$w'$ to $w$ with the reversed orientation has zero symplectic area and
zero Maslov number. Furthermore, we say that two such cylinders are
equivalent when their symplectic areas and Maslov numbers are equal.
When $w$ and $w'$ are not equivalent, we call $v$, taken up to this
equivalence relation, and also $(x,w')$ a \emph{recapping} of $(x,w)$
and write $(x,w')=:(x,w)\# v$. We usually suppress a capping in the
notation. One can assign a well-defined Maslov index to a capped path
from $L$ to $L'$ by fixing also a trivialization of $TM$ along
$x_\fc$.

For the sake of simplicity, the ground field $\F$ throughout the paper
is $\F_2$. (However, in the Hamiltonian setting of Corollary
\ref{cor:A} and Theorems \ref{thm:B} and \ref{thm:C} we can work with
any ground field $\F$.) The Floer complexes we mainly consider are
finite-dimensional vector spaces over the ``universal'' \emph{Novikov
  field} $\Lambda$ formed by formal sums
\begin{equation}
  \label{eq:lambda}
\lambda=\sum_{j\geq 0} f_j T^{a_j},
\end{equation}
where $f_j\in \F$ and $a_j\in\R$ and the sequence
$a_j$ (with $f_j\neq 0$) is either finite or $a_j\to\infty$.

The Floer complex $\CF(L,L')$, or just $\CC$ for brevity, is generated
by the intersections $L\cap L'$ with arbitrarily fixed cappings. The
differential $\p_{\Fl}$ is defined by the standard formula counting
holomorphic strips $u$ with boundary components on $L$ and $L'$,
asymptotic to the intersections with now all possible cappings and the
symplectic area $\omega(v)$ of the recapping $v$ contributing the term
$T^{\omega(v)}$ to the differential.

To be more precise, denote the generators by $x_i$. Thus
\begin{equation}
  \label{eq:CC}
\CC=\bigoplus_i \Lambda x_i.
\end{equation}
Then, assuming that the underlying almost complex structure is
regular, we have
\begin{equation}
  \label{eq:d-Fl1}
\p_{\Fl} x_i=\sum_j \lambda_{ij} x_j,
\end{equation}
where
\begin{equation}
  \label{eq:d-Fl2}
\lambda_{ij}=\sum_v f_v T^{\omega(v)}\in \Lambda
\end{equation}
Here the sum extends over all recappings $v$ of $x_j$ such that the
Maslov index difference of $x_i$ and $x_j\# v$ is 1 and $f_v\in \F$ is
the parity of the number of holomorphic strips $u$ asymptotic to $x_i$
and $x_j$, and such that the original fixed capping of $x_i$ is
equivalent to the capping obtained by attaching $x_j\# v$ to these
strips.

This complex is not graded, due to our choice of the Novikov field, or
only $\Z_2$-graded.  We denote the Floer homology, i.e., the homology
of $\CC$, by $\HF(L,L')$ or just $\HF(L)$. This is also a
finite-dimensional vector space over $\Lambda$.

The complex $\CC$ and its homology $\HF(L,L')$ split into a direct sum
of complexes $\CC_c=\CF_\fc(L,L')$ and homology groups $\HF_\fc(L,L')$
over $\fc \in \pi_1(M;L,L')$.

\subsubsection{Action filtration}
\label{sec:filtration}
To define the action filtration on $\CC=\CF(L,L')$ it is beneficial to
look at this complex from a different perspective and this is where
the condition that $L$ and $L'$ are Hamiltonian isotopic, which has
not been explicitly used so far, becomes essential. Namely, applying
the inverse Hamiltonian isotopy $\big(\varphi_F^t\big)^{-1}$ to paths
from $L$ to $L'$ we obtain paths from $L$ to itself and thus a
homeomorphism between $\PP(L,L')$ and $\PP(L,L)$ and a bijection
between $\pi_1(M;L,L')$ and $\pi_1(M;L) :=\pi_1(M;L,L)$. The
intersections $L\cap L'$ turn into Hamiltonian chords from $L$ to
itself for $\big(\varphi_F^t\big)^{-1}$. Likewise, the reference paths
$x_\fc$ become the reference paths $y_\fc$ from $L$ to $L$ and a
capping of $y\in \PP(L,L)$ is a homotopy from $y\in\fc$ to $y_\fc$
with end-points on $L$ up to the same equivalence relation.

This procedure turns the Cauchy--Riemann equation into the Floer
equation
$$
\frac{\p u}{\p s}+J\frac{\p u}{\p t}=-\nabla H_t(u)
$$
where $(s,t)\in\R\times [0, 1]$ and $u\colon \R\times [0,1]\to M$, and
$H_t= -F_t\circ \varphi_F^t$ is a Hamiltonian generating the inverse
of the isotopy $\varphi_F^t\colon M\to M$ from $L$ to
$L'$. Holomorphic strips with boundaries on $L$ and $L'$ and
asymptotic to $L\cap L'$ become solutions $u$ of the Floer equation
with boundary on $L$ asymptotic to Hamiltonian chords. The Floer
differential now counts solutions of the Floer equation. The
Hamiltonian action of a capped path $\by=(y,w)$ is given by the
standard formula
\begin{equation}
  \label{eq:action}
\CA(\by)=-\int_w\omega+ \int_0^1 H_t(y(t))\,dt.
\end{equation}

The Floer differential is strictly action-decreasing, and we obtain
the required action filtration on $\CC$ with the convention that for
$\lambda\in\Lambda$ given by \eqref{eq:lambda}, we have
\begin{equation}
\label{eq:filt1}
\nu(\lambda):=\min \{ a_j\}, \text{ where } f_j\neq 0,
\end{equation}
and
\begin{equation}
\label{eq:filt2}
\CA\left(\sum \lambda_i y_i\right):=\max_i\CA(\lambda_i y_i),
\text{ where } \CA(\lambda_i y_i):=\CA(y_i)-\nu(\lambda_i). 
\end{equation}

A few remarks are due at this point. First of all, the action
filtration depends on the choice of the reference paths $x_\fc$ (or
$y_\fc$). Namely, on every direct summand $\CC_\fc$ a change of the
reference path shifts the filtration by a constant. Thus the
filtration is well-defined only up to independent shifts on these
summands. Clearly, this ambiguity does not affect the barcode of $\CC$
introduced in Section \ref{sec:barcodes-def}. In particular, the bar
length and the number of bars of a given length and $b_\eps(L,L')$ are
independent of the choice of reference paths.

Furthermore, while the generators of $\CC$ are (capped) intersections,
the differential genuinely depends on the background almost complex
structure or, in the second interpretation, on the almost complex
structure and the Hamiltonian $F$. However, the resulting complexes
are chain homotopy equivalent with homotopy preserving the filtration
up to a shift. Thus, the number of bars of a given length either and
$b_\eps(L,L')$ are again well-defined.

The Floer homology $\HF(L,L')$, as a vector space over $\F$, inherits
the action filtration from $\CC$. To be more specific, we have a
family of vector spaces $\HF^a(L,L')$ over $\F$, parametrized by
$a\in \R$, where $\HF^a(L,L')$ is the homology of the subcomplex of
$\CC$ comprising the chains $\sum \lambda_i y_i$ with action less than
$a$. The inclusion of subcomplexes for $a<a'$ gives rise to natural
maps $\HF^a(L,L')\to \HF^{a'}(L,L')$.

Essentially by construction, we have the ``Poincar\'e
duality''. Namely, for a suitable choice of the auxiliary data,
\begin{equation}
  \label{eq:PD0}
\CF(L',L)=\CF(L,L')^* 
\end{equation}
over $\Lambda$ with ``inverted'' filtration.  (Here we used the fixed
basis $\{x_i\}$ to identify the vector spaces $\CF(L',L)$ and
$\CF(L,L')$ and their duals, and then the Floer differential in
$\CF(L',L)$ turns into the adjoint of the Floer differential in
$\CF(L,'L)$.)  As a consequence, $\HF(L',L)=\HF(L,L')^*$.

We also note that whether or not $L$ and $L'$ are transverse, only a
finite collection of elements in $\pi_1(M;L,L')\cong\pi_1(M;L)$ is
represented by Hamiltonian chords. Thus $\CC_{\fc}=0$ for all but a
finite collection of $\fc\in\pi_1(M;L,L') \cong\pi_1(M;L)$; this
collection might however depend on the Hamiltonian isotopy; cf.\
\cite[Prop.\ 6.2]{Us}.

\begin{Remark} In the situation we are interested in there is usually
  a natural choice of the Hamiltonian isotopy. Namely, when
  $L'=\varphi^k(L)$ as in Section \ref{sec:def}, the isotopy comes
  from $\varphi_H^t$ with time interval $[0, k]$ scaled to $[0, 1]$.
  In Section \ref{sec:def2}, $L=L_0$ and $L'=\varphi^k(L_1)$, where
  $L_0$ and $L_1$ are assumed to be Hamiltonian isotopic, and the
  isotopy from $L$ to $L'$ is obtained by concatenating the two
  isotopies.
\end{Remark}

\begin{Remark}[Exact Lagrangians, I]
  \label{rmk:filtr-other}
  Finally, note that there are of course other settings where the
  filtration on $\CC$ is defined and has the desired properties. One
  is when $(M,\omega)$ is an exact symplectic manifold (i.e., $\omega$
  is exact) meeting certain additional requirements at infinity, and
  $L$ and $L'$ are exact Lagrangian submanifolds (i.e., the
  restrictions of a primitive of $\omega$ to $L$ and $L'$ are exact),
  not necessarily Hamiltonian isotopic. Then $\Lambda=\F$. In this
  case, sometimes we can even allow one of the manifolds to be
  non-compact. For instance, $L$ can be Hamiltonian isotopic to the
  zero section in a cotangent bundle of a closed manifold and $L'$ can
  be a fiber.
\end{Remark}

\subsubsection{Period group, the action spectrum and the action
  difference set}
\label{sec:spectrum}
In this subsection we do not assume that $L$ and $L'$ are transverse
unless explicitly stated otherwise.

For every class $\fc\in\pi_1(M;L)$, denote by $\Gamma_\fc$ the
subgroup of $\R$ formed by the integrals $\omega(v)$ for all maps
$v\colon \AA=S^1\times [0,1]\to M$ sending $\p\AA$ to $L$ and with the
meridian $v(\{0\}\times [0,1])$ in $\fc$. (Such annuli $v$, with $\fc$
fixed, taken up to homotopy or the equivalence relation discussed
above, form a group of recappings of paths in $\fc$ and $\Gamma_\fc$
is the group of action changes resulting from recappings.)  The
monotonicity condition forces this group to be discrete for the
``unit'' class $\fc=1$ of the constant path. In general, $\Gamma_\fc$
depends on $\fc$ and can be discrete for some classes $\fc$ and dense
for other classes; see Example \ref{ex:Gamma}. In the former case we
denote by $\lambda_\fc$ its positive generator and call $\fc$
\emph{toroidally rational}. As usual, we set $\lambda_\fc=\infty$ when
$\fc$ is \emph{atoroidal}, i.e., $\Gamma_\fc=\{0\}$.  Finally, let
$\Gamma\subset \R$ be the subgroup generated by the union of all
$\Gamma_\fc$. We will refer to $\Gamma_\fc$ as a \emph{period group}
and to $\Gamma$ as the \emph{complete period group}. (Hypothetically
it is possible that $\Gamma$ is dense even when all $\Gamma_\fc$ are
discrete, but we do not have an example when this happens.)

The \emph{action spectrum} $\CS_\fc(L,L')\subset \R$ (for a class
$\fc$) is the collection of actions for all capped chords in $\fc$ or,
equivalently, capped intersections $L\cap L'$ in the class
$\fc$. Clearly, a change of a reference chord $y_\fc$ results in a
shift of $\CS_\fc(L,L')$ by a constant. Furthermore, as we have
already mentioned, only a finite collection of free homotopy classes
can contain Hamiltonian chords, and hence $\CS_\fc(L,L')=\emptyset$
for all but a finite set of the classes $\fc$. In general, this
collection of classes depends on the Hamiltonian isotopy from $L$ to
$L'$, and a change of the isotopy results in action shifts and
``relabeling'' of the action spectra; cf. \cite[Prop.\ 6.2]{Us}. We
denote by $\CS(L,L')$ the union of the sets $\CS_\fc(L,L')$ for all
$\fc$.

The following is a list of standard properties of the action spectrum
which are relevant to our purposes although not directly used in the
proofs:
\begin{itemize}

\item The sets $\CS_\fc(L,L')$ are non-empty for only a finite
  collection of elements $\fc\in \pi_1(M;L)$.

\item The sets $\CS_\fc(L,L')$, and hence their union $\CS(L,L')$,
  have zero measure.

\item The set $\CS_\fc(L,L')$ is countable whenever the intersections
  $L\cap L'$ in the class $\fc$ are transverse. (The converse is not
  true.)

\item The set $\CS_\fc(L,L')$ is invariant under translations by
  elements of $\Gamma_\fc$.
    
\item The set $\CS_\fc(L,L')\neq \emptyset$ is compact if and only if
  $\Gamma_\fc=\{0\}$.

\item The set $\CS_\fc(L,L')\neq \emptyset$ is closed if and only if
  $\Gamma_\fc$ is discrete.

\item The set $\CS_\fc(L,L') \neq \emptyset$ is dense if and only if
  $\Gamma_\fc$ is dense.
    
\end{itemize}

We denote by $\CalD_\fc (L,L')\subset \R$ the set comprising all
action differences between the capped intersections in $\fc$, i.e.,
$$
\CalD_\fc(L,L'):=\CS_\fc(L,L')-\CS_\fc(L,L') = \{a-b\mid a, b\in
\CS_\fc(L,L')\}.
$$
This set is countable whenever the intersections $L\cap L'$ in the
class $\fc$ are transverse, and dense when $\Gamma_\fc$ is dense. (The
converse is not true in both cases.) Furthermore, $\CalD_\fc(L,L')$ is
closed if and only if $\CS_\fc(L,L')$ is closed.  An important point
is that in general $\CalD_\fc (L,L')$ need not be a zero measure
set. (For instance, $\CalD_\fc (L,L')=[-1,1]$ when $\CS_\fc(L,L')$ is
the standard Cantor set; see, e.g., \cite[p.\ 87]{GO}.)  We also let
$\bar{\CalD}_\fc(L,L')$ to be the closure of $\CalD_\fc (L,L')$ and
$\bar{\CalD}(L,L')$ stand for the union of the closures
$\bar{\CalD}_\fc(L,L')$ for all $\fc\in\pi_1(M;L)$. This union is also
closed since $\CalD_\fc (L,L')\neq \emptyset$ only for a finite
collection of classes $\fc$.

Applying these constructions to the Hamiltonian diffeomorphism
$\id\times \varphi$ of the symplectic square
$\big(M\times M, -\omega \oplus \omega\big)$, we obtain the filtered
Floer complex $\CC:=\CF(\varphi)$ of $\varphi$. This complex is
isomorphic to the standard $\Z_2$-graded Floer complex over $\Lambda$
generated by 1-periodic orbits in all free homotopy classes
$\fc\in\tpi_1(M)$. This distinction from the more customary
construction limited to contractible orbits is essential; see Section
\ref{sec:discussion}. The free homotopy class of a $k$-periodic orbit
of $\varphi_H^t$ is completely determined by the Hamiltonian
diffeomorphism $\varphi_H$ and is independent of the isotopy. This is
a consequence of the proof of Arnold's conjecture which in the
monotone case was established in \cite{Fl89}.  As we have already
pointed out, in the Lagrangian setting the situation is more
complicated.  The period group $\Gamma_\fc\subset \R$ is formed by the
integrals $\omega(v)$ where now $v\colon \T^2=S^1\times S^1\to M$ with
the meridian $v(\{0\}\times S^1)$ in $\fc$, and $\lambda_\fc$ is the
postive generator of $\Gamma_\fc$ when this groups is non-trivial and
$\lambda_\fc=0$ otherwise. (The homotopy classes of such tori form a
group.) As above, we denote by $\Gamma$ the complete period group,
i.e., the group generated by the union of the groups $\Gamma_\fc$.

\begin{Example}
  \label{ex:Gamma}
  Assume first that $M=\T^2$ with area $a$ and let as above
  $\tpi_1(M)$ be the set of free homotopy classes of loops in $M$,
  i.e., conjugacy classes in $\pi_1(M)$. Then
  $\tpi_1(M)=\pi_1(M)=\Z^2$ and $\lambda_\fc=a$ for every primitive
  class $\fc\neq 0$ and, using additive notation,
  $\lambda_{k\fc}=k\lambda_\fc$ for every $k\geq 0$; see \cite{Or}. A
  similar description applies to $\T^{2n}=\T^2\times\cdots\times \T^2$
  equipped with the standard symplectic structure for which all
  factors $\T^2$ have the same area $a$, although now
  $\lambda_\fc\in a\Z$ may depend on $\fc$ even when $\fc$ is
  primitive. A surface $\Sigma$ of genus $g\geq 2$ or, more generally,
  a symplectically hyperbolic manifold (see, e.g., \cite{Ke,Po}) is
  atoroidal, i.e,. $\Gamma_\fc=0$ for all $\fc$.  Next, let
  $M=\T^2\times\Sigma$.  Then $\lambda_\fc$ is completely determined
  by the projection of $\fc$ to $\T^2$ and then calculated as in the
  previous example.  Finally, let $M=\T^2_1\times\T^2_a$, where the
  first torus has area 1 and the second has an irrational area
  $a$. Then classes $\fc$ lying in $\pi_1(\T^2_1)$ and $\pi_1(\T^2_a)$
  are toroidally rational and $\lambda_\fc$ is calculated as
  above. For other classes, the group $\Gamma_\fc$ is dense in $\R$.
\end{Example}

\begin{Remark}[Growth of the action spectrum]
  \label{rmk:spec} As we have pointed out in the introduction, one
  difficulty in thinking about topological entropy in terms of Floer
  theory is that in many cases the ``effective'' diameter of the
  action spectrum $\CS\big(\varphi^k\big)$ grows sub-exponentially
  with $k$ even when the cardinality of $\CS\big(\varphi^k\big)$ grows
  super-exponentially (see Remark \ref{rmk:PO-Growth}). For instance,
  assume that $\omega$ is atoroidal. Then it is not hard to show that
  for a suitable choice of reference loops in $\tpi_1(M)$ the diameter
  of $\CS\big(\varphi^k\big)$ grows sub-exponentially when the
  pullback of $\omega$ to the universal covering of $M$ has a
  primitive with sub-exponential growth. For instance, when $M$ is a
  surface of genus $g\geq 2$ or, more generally, $M$ is symplectically
  hyperbolic, the diameter grows linearly.

  When $M$ is simply connected and monotone, one can pin every finite
  bar to be contained in the interval
  $\big[0,\lambda_M+\|\varphi^k\|_{\hn}\big]\subset
  \big[0,\lambda_M+k\|\varphi\|_{\hn}\big]$, where $\lambda_M$ is the
  rationality constant; see \cite{Us} and Remark
  \ref{rmk:pinned-bars}.  (Moreover, under these assumptions on $M$,
  we can replace the Hofer norm by the $\gamma$-norm resulting in a
  bounded interval in some instances, e.g., for $\CP^n$; see
  \cite{EP,KS}.) A similar upper bound holds for $M=\T^{2n}$, but now
  one has to also use the result from \cite{Or} mentioned in
  Example~\ref{ex:Gamma}.
  \end{Remark}

\subsubsection{Floer package and shrinking the Novikov field}
\label{sec:Floer-pck}

Purely formally, the above constructions can be summarized as the
following ``Floer package'':
\begin{itemize}

\item The finite-dimensional vector space $\CC$ over $\Lambda$ with a
  fixed set of generators $x_i$; see~\eqref{eq:CC}.

\item The differential $\p_{\Fl}$; see \eqref{eq:d-Fl1}.

\item The action filtration on $\CC$ given by \eqref{eq:filt1} and
  \eqref{eq:filt2} such that $\p_{\Fl}$ is strictly action decreasing.

\end{itemize}
We emphasize that in this package there are no algebraic constraints
on the actions $\CA(x_i)$ other than that $\p_{\Fl}$ is
\emph{required} to be strictly action decreasing.

Next, observe that all exponents occurring in $\lambda_{ij}$ in
\eqref{eq:d-Fl2} are in $\Gamma$ or in $\Gamma_\fc$ when we focus on
$\CC_\fc$. Thus we could have replaced in the construction the
universal Novikov field $\Lambda$ by the field $\Lambda^\Gamma$
defined in a similar fashion, but with all exponents $a_j$ in
\eqref{eq:lambda} in $\Gamma$ (or in $\Gamma_\fc$). Moreover, we could
have worked with the field $\Lambda^G$ for any subgroup $G$ of $\R$
containing $\Gamma$, resulting in the same barcode. (For instance,
$\Lambda=\Lambda^{\R}$.) When essential, we will write
$\CC(\Lambda^\Gamma)$ to indicate the Novikov field.

On a purely formal level of a Floer package, we can define $\Gamma$ as
a countable subgroup of $\R$ generated by all exponents occurring in
$\lambda_{ij}$ and then use the field $\Lambda^G$ whenever
$\Gamma \subset G$. (This group $\Gamma$ can potentially be smaller
that the period group $\Gamma$ defined in Section \ref{sec:spectrum}
geometrically.)

Our choice to mainly work with $\Lambda$ rather than $\Lambda^\Gamma$
is dictated by expository considerations: it is convenient to have the
Novikov ring independent of the geometrical setting.

\subsection{Persistent homology and barcodes}
\label{sec:barcodes}

\subsubsection{Barcodes}
\label{sec:barcodes-def}
In this section we briefly recall a few basic facts and definitions
concerning persistent homology and barcodes in the context of Floer
theory. We refer the reader to \cite{PRSZ} for a very detailed
introduction and a discussion in much broader context. Here, treating
barcodes in the framework of Lagrangian Floer complexes, we closely
follow \cite{UZ}, with some minor simplifications. (A different
although equivalent approach to barcodes in this context is put forth
in \cite{KS}. That approach is, however, slightly less convenient for
our purposes.)

The key difference between the settings used here and in \cite{UZ} is
that Floer complexes in \cite{UZ} are $\Z$-graded, while in this
paper, as in \cite{CGG}, the complexes are ungraded (or
$\Z_2$-graded), i.e., in \cite{UZ} the differential comprises maps
between different spaces over $\Lambda$, while here $\p_{\Fl}$ is a
map from $\CC$ to itself. However, the constructions, results and
proofs from \cite{UZ} carry over to our framework. For instance, one
can apply these results to the ``two-storey'' $\Z$-graded complex
$$
0\to\CC/\ker \p_{\Fl}\stackrel{\p_{\Fl}}{\longrightarrow} \ker
\p_{\Fl}\to 0
$$
over $\Lambda$.

As in Section \ref{sec:Lagr}, consider the filtered Lagrangian Floer
complex $\CC:=\CF(L,L')$ of a transverse pair of closed monotone
Lagrangian submanifolds $L$ and $L'$.  A finite set of vectors
$\xi_i\in \CC$ is said to be \emph{orthogonal} if for any collection
of coefficients $\lambda_i\in\Lambda$ we have
\[
  \CA\big(\sum\lambda_i\xi_i\big)=\max \CA(\lambda_i\xi_i),
\]
where $\CA$ is defined by \eqref{eq:filt1} and \eqref{eq:filt2}. It is
not hard to show that an orthogonal set is necessarily linearly
independent over $\Lambda$.

\begin{Definition}
  \label{def:sing_decomp}
  A basis $\Sigma=\{\alpha_i, \,\eta_j, \,\gamma_j\}$ of $\CC$ over
  $\Lambda$ is said to be a \emph{singular value
      decomposition} if
\begin{itemize}
\item $\p_{\Fl}\alpha_i=0$,
\item $\p_{\Fl}\gamma_j=\eta_j$,
\item the basis is orthogonal.
\end{itemize}
\end{Definition}

It is shown in \cite[Sections 2 and 3]{UZ} that $\CC$ admits a
singular value decomposition. Ordering the pairs $(\eta_j,\,\gamma_j)$
by the action difference, we have
$$
\CA(\gamma_1)-\CA(\eta_1)\leq\CA(\gamma_2)-\CA(\eta_2)\leq\ldots .
$$
This increasing sequence together with $\dim_\Lambda\HF(L,L')$
infinite bars (corresponding to the basis elements $\alpha_i$) is
referred to as the \emph{barcode} of $\CC$ and denoted by $\CB(\CC)$
or $\CB(L,L')$.  Fixing a class $\fc$ we also obtain the barcode
$\CB_\fc(L,L')$.  This is a multiset. As a consequence of the
discussion in Section \ref{sec:Floer-pck},
$\CB_\fc(L,L')\subset \CalD_\fc(L,L')\cup\{\infty\}$ if we ignore
multiplicities. The barcode is independent of the choice of a singular
value decomposition and auxiliary data involved in the construction of
$\CC$; see, e.g., \cite{UZ}. As in \eqref{eq:b-eps}, we set
\begin{equation}
  \label{eq:b-eps2-tr}
  b_\eps(L,L')=b_\eps(\CC):=
  \big|\{\beta\in \CB(\CC)\mid \beta>\eps\}\big|.
\end{equation}

The complexes $\CF(L',L)=\CF(L,L')^*$ and $\CF(L,L')$ have the same
barcode:
$$
\CB(L,L')=\CB(L',L);
$$
cf.\ \cite{UZ}. (This does not follow directly from \eqref{eq:PD0}
because a singular value decomposition for $\CF(L,L')$ need not be a
singular value decomposition for $\CF(L',L)$ when the two spaces are
identified by fixing a basis of capped intersections.) Hence,
\begin{equation}
  \label{eq:PD}
b_\eps(L,L')=b_\eps(L',L).
\end{equation}
We also set $b(\CC)=b(L,L'):=|\CB(L,L')|$ to be the
total number of bars in the barcode. Then
$$
|L\cap L'|=\dim_\Lambda\CF(L,L') = 2b(L,L')-\dim_\Lambda\HF(L,L')\geq
b(L,L').
$$
In particular, $b_\eps(L,L')$ gives a lower bound for the number of
intersections:
\begin{equation}
\label{eq:intersections-b}
|L\cap L'| \geq b_\eps(L,L').
\end{equation}
We emphasize that here and throughout this section we have assumed
that $L\pitchfork L'$.

Clearly,
$$
\CB\big(\psi(L),\psi(L')\big)=\CB(L,L')
$$
for any symplectomorphism $\psi\colon M\to M$, and hence 
 \begin{equation}
  \label{eq:inv}
b_\eps\big(\psi(L),\psi(L')\big)=b_\eps(L,L').
\end{equation}
Furthermore, $b_\eps(L,L')$ is constant as a function of $\eps$ on any
interval in the complement of $\bar{\CalD}(L,L')$.

The longest finite bar $\beta_{\max}(L,L')$ in $\CB(L,L')$ is called
the \emph{boundary depth}. As shown in \cite{Us},
$\beta_{\max}(L,L')\leq d_{\hn} (L,L')+\beta_{\max}(L,L)$ and, in
particular, $\beta_{\max}(\varphi)\leq \|\varphi\|_{\hn}$. (In
general, $\beta_{\max}(L,L)$ can be non-zero. For instance, for a
displaceable circle $L$ bounding a disk of area $a$ in a surface,
$\CB(L,L)$ has only one bar, which has length $a$, and hence
$\beta_{\max}(L,L)=a>0$.)

\begin{Remark}[Pinned bars]
  \label{rmk:pinned-bars}
  Our definition of a barcode is a simplification of the standard one
  in which, when $\Gamma=\{0\}$, a barcode is the collection of the
  intervals $[\CA(\eta_j), \CA(\gamma_j)]$ and
  $[\CA(\alpha_i),\infty)$ rather than just their length; see, e.g.,
  \cite{PRSZ}. Hence, a bar is ``pinned'' to its beginning in
  $\R$. When working over the universal Novikov field $\Lambda$, the
  situation is similar, but now $\R$ is replaced by
  $\R/\Gamma$. Namely, with more care in the choice of a singular
  value decomposition, a finite bar is determined by $\CA(\gamma_i)$,
  viewed as an element of $\R/\Gamma$, and the difference
  $\CA(\gamma_i)-\CA(\eta_i)$, and an infinite bar is the pair
  $\CA(\alpha_i)\in \R/\Gamma$ and $\infty$; see \cite{UZ}. Thus a bar
  is ``pinned'' to a point in $\CS(L,L')/\Gamma$. However, then the
  position of each bar depends on the auxiliary data (reference paths
  and the Hamiltonian isotopy from $L$ to $L'$), and the shift
  resulting from a change of the data also depends on the connected
  component $\fc\in \pi_1(M; L,L')$. Since here we are mainly
  concerned with counting bars of length greater than $\eps$, the
  simplified definition is more convenient for our purposes.
 \end{Remark}

 The barcodes are fairly insensitive to small perturbations of the
 Lagrangian submanifolds with respect to the Hofer metric. Namely,
 assume that $d_{\hn}(L',L'')<\delta/2$. Then
\begin{equation}
  \label{eq:insensitive1}
b_{\eps+\delta}(L,L') \leq b_\eps(L,L'')\leq b_{\eps-\delta}(L,L').
\end{equation}
We refer the reader to \cite{KS, PRSZ, UZ} for the proof.

These constructions apply \emph{verbatim} to the direct summand
$\CF_\fc(L,L')$ for $\fc\in \pi_1(M;L)$.

It is not hard to extend the notion of barcode to the situation where
$L$ and $L'$ are not transverse. However, here we are interested only
in counting bars of length above a fixed treshold and we extend its
definition to the non-transverse case in an \emph{ad hoc} manner in
the next section.

Specializing to the case where $L$ is the diagonal in $M\times M$ and
$L'$ is the graph of a strongly non-degenerate Hamiltonian
diffeomorphism $\varphi\colon M\to M$, we obtain the barcode
$\CB(\varphi)$ of $\CF(\varphi)$ (for all free homotopy classes of
loops $\fc\in\tpi_1(M)$), which is again insensitive to small
perturbations of $\varphi$ with respect to the Hofer metric and hence
to $C^\infty$-small perturbations, and
$\CB_\fc(\varphi)\subset \CalD_\fc(\varphi)\cup\{\infty\}$. Likewise,
$\bar{\CalD}(\varphi)$ will stand for the union of the closures
$\bar{\CalD}_\fc(\varphi)$, etc.

Throughout this section we could have replaced the universal Novikov
field $\Lambda$ by a smaller field $\Lambda^\Gamma$ as in Section
\ref{sec:Floer-pck} with exactly the same resulting barcode with the
same properties. In particular, this change would not affect
$b_\eps(L,L')$. Moreover, it would lead to a small technical
advantage. Namely, assume that all intersections $L\cap L'$ have
actions distinct modulo $\Gamma$, i.e., $\CA(x_i)\neq \CA(x_j)$,
$i\neq j$, in $\CS(L,L')/\Gamma $. (This is a $C^\infty$-generic
condition.) For $\xi\in\CC(\Lambda^\Gamma)$, write
$$
\xi= f T^a x +\ldots,
$$
where the dots stand for lower-action terms and $x$ is one of the
capped intersections $x_i$ (and $f\neq 0$). It is easy to see that
since the intersections have distinct actions modulo $\Gamma$ and
$a\in \Gamma$, the term $T^a x$ and the intersection $x$ are
unique. We will refer to $T^a x$, or sometimes just $x$, as the
\emph{leading action term} in $\xi$.

\begin{Example}
  \label{exam:orthogonality}
  Assume that the intersections have distinct actions modulo $\Gamma$.
  Then it is easy to see that the set $\xi_i\in\CC(\Lambda^\Gamma)$ is
  orthogonal if and only if the leading action terms $x_i$ are
  distinct.
\end{Example}

\subsubsection{Bounding $b_\eps(\CC)$ from below.}
\label{sec:isolated}
The proof of Theorem \ref{thm:B} hinges on a lower bound on
$b_\eps(\varphi)$ via the number of periodic orbits which are in a
certain sense energy-isolated. In this section we will deal with the
algebraic aspect of the argument.

It is convenient, although strictly speaking not necessary, to
introduce the \emph{Floer graph} $G$ associated with a Floer package
$\CC$; cf.\ \cite{CGG}. The vertices of $G$ are the generators
$x_i$. For each non-zero term $f T^a $ in $\lambda_{ij}$ (see
\eqref{eq:d-Fl2}) we connect $x_i$ to $x_j$ by an arrow and label that
arrow by the exponent $a$. (Thus $G$ is a directed graph with finitely
many vertices but possibly infinitely many edges.) The length of an
arrow is the action difference
$\CA(x_i)-\CA(T^ax_j)=\CA(x_i)-\CA(x_j)+a$, i.e., the energy of any
underlying Floer trajectory. We say that $x=x_i$ is $\eps$-isolated if
every arrow from or to $x$ has length strictly greater that
$\eps$. For instance, $x$ is $\eps$-isolated if every Floer trajectory
asymptotic to $x$ at $\pm\infty$ has energy strictly greater than
$\eps$. (The converse need not be true.)

\begin{Proposition}
  \label{prop:isolated}
  Assume that $G$ has $p$ $\eps$-isolated vertices. Then
  $b_\eps(\CC) \geq p/2$.
\end{Proposition}

\begin{proof} Let us switch from the universal field $\Lambda$ to
  $\Lambda^\Gamma$.  Thus, throughout the proof,
  $\CC:=\CC(\Lambda^\Gamma)$. By continuity, we can assume in addition
  that all generators $x_i$ have distinct actions modulo
  $\Gamma$. Indeed, a small change of actions $\CA(x_i)$ does not
  affect $b_\eps(\CC)$ and $\eps$-isolation, as a consequence of
  \eqref{eq:insensitive1} and the fact that in the definition of
  $b_\eps$, \eqref{eq:b-eps2-tr}, the bars are required to be strictly
  greater than $\eps$.  Furthermore, as has been pointed out in
  Section \ref{sec:barcodes-def} this can be achieved by a
  $C^\infty$-small perturbation. (In fact, since we are working in a
  purely algebraic setting we can simply change the filtration of
  $\CC$ by altering $\CA(x_i)$ in the formal framework of a Floer
  package; see Section \ref{sec:Floer-pck}.) Now Example
  \ref{exam:orthogonality} applies.

  Next, recall from \cite{UZ} that a non-zero element
  $\zeta\in\im \p_{Fl}$ is said to be \emph{$\eps$-robust} if
  $\CA(\xi)- \CA(\zeta)>\eps$ for every $\xi$ with
  $\p_{\Fl}\xi=\zeta$. A subspace $W\subset \im \p_{Fl}$ is said to be
  $\eps$-robust if every (non-zero) vector $\zeta\in W$ is
  $\eps$-robust. The key fact established in \cite{UZ} that we will
  use in this proof is that the number of finite bars of length
  greater than $\eps$ is equal to the maximal dimension of an
  $\eps$-robust subspace. Thus it suffices to find an $\eps$-robust
  subspace $W$ with $\dim W \geq p/2-\dim\H(\CC)$.

  Let $x_1, \ldots, x_p$ be the $\eps$-isolated vertices and let $V$
  be their span.  We will need the following two observations:

  \begin{itemize}

  \item[\reflb{O1}{\rm{(i)}}] For every linear combination $\xi$ of
    $\eps$-isolated vertices $x_i$, we have
    $\CA(\xi)-\CA(\p_{\Fl}\xi)>\eps$.

  \item[\reflb{O2}{\rm{(ii)}}] Let $w_j=x_j+\ldots$, where the dots
    stand for lower action terms, be any collection of vectors with
    $x_j$ distinct, i.e., $w_j$ are orthogonal. Then any exact (i.e.,
    in the image of $\p_{\Fl}$) linear combination of $w_j$ is
    $\eps$-robust.

  \end{itemize}

Next, from the short exact sequence
$$
0\to \ker (\p_{\Fl}|_V) \to V \to \p_{\Fl}(V) \to 0,
$$
we have
$$
\dim V= \dim \ker (\p_{\Fl}|_V) + \dim \p_{\Fl}(V) .
$$
Let $Y$ be an orthogonal complement of $\ker (\p_{\Fl}|_V)$ in
$V$. (We refer the reader to \cite[Sec.\ 2]{UZ} for an extensive
  discussion of orthogonality in the nonarchimedean setting and
  further references.) Thus $V=\ker(\p_{\Fl}|_V)\oplus Y$ and
$\p_{\Fl}$ induces an isomorphism between $Y$ and $\p_{\Fl}(V)$.

Let us now modify $Y$. If all elements of $\p_{\Fl}(V)$ are
$\eps$-robust we do nothing. Assume not: there is a vector
$\xi_1\in Y$ such that $\p_{\Fl}\xi_1$ is not $\eps$-robust, but
$\CA(\xi_1)-\CA(\p_{\Fl}\xi_1)>\eps$ by Observation
\ref{O1}. Therefore, there exists $\zeta_1$ with
$\p_{\Fl}\zeta_1=\p_{\Fl} \xi_1$ and
$$
\CA(\zeta_1)-\CA(\p_{\Fl}\xi_1)<\eps.
$$
Then $w_1:=\xi_1 - \zeta_1$ is closed and $\CA(w_1)=\CA(\xi_1)$,
i.e., the leading term in $w_1$ is the leading term in $\xi_1$.  Let
$Y_1\subset Y$ be an orthogonal complement to $\xi_1$. If every
element of $\p_{\Fl}(Y_1)$ is $\eps$-robust the process stops. If not,
we pick $\xi_2$ such that $\p_{\Fl}\xi_2$ is not $\eps$-robust, etc.
Proceeding, we will construct, for some $s\geq 1$, vectors
$\xi_1, \ldots, \xi_s$ in $Y$ as above together with vectors $w_i$,
and an orthogonal complement $Y_s$ to $\Span(\xi_1,\ldots,\xi_s)$ in
$Y$ such that $\p_{\Fl}(Y_s)$ is $\eps$-robust and
$\dim \p_{\Fl}(Y_s)=\dim Y_s$.

Each vector $w_i$ has the same leading term as $\xi_i$, which we can
assume
to be one of $\eps$-isolated generators
$x_j$.  The vectors $w_i$ are still orthogonal to $\ker (\p_{\Fl}|_V)$
and to each other since so are the vectors $\xi_i$. Hence, by Example
\ref{exam:orthogonality}, the leading terms $x_j$ of these vectors are
distinct and also distinct from the leading terms in any orthogonal
basis in $\ker(\p_{\Fl}|_V)$. Hence, Observation \ref{O2} applies to a
an orthogonal basis in
$$
W_0=\ker(\p_{\Fl}|_V)\oplus \Span(w_1,\ldots, w_s).
$$
Also set
$$
W_1=\p_{\Fl}(Y_s).
$$
Every element in $W_0$ is closed, and every vector in $W_1$ is exact
and $\eps$-robust by construction. Furthermore,
$$
\dim W_0 + \dim W_1 = \dim V= p
$$
and $W_0$ has an orthogonal basis with distinct leading terms $x_j$.

Finally, let $W_{00}$ be the subspace of exact vectors in $W_0$, i.e.,
the kernel of the natural map $W_0\to\H(\CC)$. Every element in
$W_{00}$ is $\eps$-robust by Observation \ref{O2}. Clearly,
$$
\dim W_{00}\geq \dim W_0-\dim \H(\CC) .
$$

One of the spaces $W_0$ or $W_1$ has dimension at least $p/2$. If this
is $W_1$, we set $W=W_1$ and the proof is finished. If this is $W_0$,
we set $W=W_{00}$, and then $\dim W\geq p/2-\dim \H(\CC)$ and every
element in $W$ is $\eps$-robust.
\end{proof}

\section{Definitions revisited and general properties}
\label{sec:def+prop}

\subsection{Definitions: pairs of Lagrangian submanifolds}
\label{sec:def2}
In this section we slightly generalize the definition of barcode
entropy from Section \ref{sec:def}, extending it to pairs of
Hamiltonian isotopic Lagrangian submanifolds. Thus let $M$, $L$ and
$L'$ be as in Section \ref{sec:prelim}: $M$ is a symplectic manifold,
compact or ``tame'' at infinity (e.g., convex), and $L$ and $L'$ are
closed monotone Lagrangian submanifolds Hamiltonian isotopic to each
other.

To extend the barcode counting function $b_\eps(L,L')$ to the
situation where $L$ and $L'$ need not be transverse, set
\begin{equation}
  \label{eq:b-eps3}
b_\eps(L,L'):=\liminf_{\tL'\to L'}b_\eps(L,\tL')\in \Z.
\end{equation}
Here the limit is taken over Lagrangian submanifolds
$\tL'\pitchfork L$ which are Hamiltonian isotopic to $L'$ and converge
to $L'$ in the $C^\infty$-topology (or at least in the
$C^1$-topology).  Note that, as a consequence,
$d_{\hn}(\tL',L')\to 0$. By \eqref{eq:PD}, we could alternatively
require that $\tL'$ is Hamiltonian isotopic to $L$, transverse to $L'$
and converges to $L$.  Furthermore, since $b_\eps(L,L')\in\Z$, the
limit in \eqref{eq:b-eps3} is necessarily attained, i.e., there exists
$\tL'$ arbitrarily close to $L'$ such that
$b_\eps(L,L')=b_\eps(L,\tL')$.

Moreover, when $\eps\not\in\bar{\CalD}(L,L')$, the right-hand side of
\eqref{eq:b-eps3} stabilizes before the limit, i.e.,
\begin{equation}
  \label{eq:b-eps2}
b_\eps(L,L')=b_\eps(L,\tL'),
\end{equation}
when $\tL'\pitchfork L$ and $\tL'$ is $C^\infty$-close and Hamiltonian
isotopic to $L'$, as is easy to see from \eqref{eq:insensitive1}.

With this definition, $b_\eps(L,L')$ is monotone increasing as
$\eps\searrow 0$, and \eqref{eq:PD}, \eqref{eq:inv} and
\eqref{eq:insensitive1}
continue to hold. For instance, to prove the first inequality in
  \eqref{eq:insensitive1}, note first that in \eqref{eq:b-eps3} we
  could have replaced the lower limit over $\tL'\to L'$ by the lower
  limit over $\tL\to L$. Then
  $$
  b_{\eps+\delta}(L,L'):=\liminf_{\tL\to L}b_{\eps+\delta}(\tL,L')\leq
  \liminf_{\tL\to L}b_\eps(\tL,L'')=: b_\eps(L,L''),
  $$
as desired.

Furthermore, as has been mentioned in the introduction,
$b_\eps(L, L')$ gives a lower bound on the number of transverse
intersections which is in some sense stable under small perturbations
with respect to the Hofer distance; cf.\ Remark
\ref{rmk:intersections}. To be more precise, assume that Lagrangian
submanifolds $L$, $L'$ and $L''$ are Hamiltonian isotopic,
$L''\pitchfork L$ and $d_{\hn}(L',L'')<\delta/2$. Then, regardless of
whether $L$ and $L'$ are transverse or not, we have
  \begin{equation}
    \label{eq:intersections}
  |L\cap L''|\geq b_{\eps}(L,L'')\geq b_{\eps+\delta}(L,L').
\end{equation}
Here the first inequality follows from \eqref{eq:intersections-b}, and
in the second we use \eqref{eq:insensitive1} and the fact that
$d_{\hn}(L',L'')<\delta/2$. These inequalities play a central role in
the proof of Theorem~\ref{thm:A}.

Let now $\varphi=\varphi_H\colon M\to M$ be a compactly supported
Hamiltonian diffeomorphism. Similarly to Definition
\ref{def:hbr-rel1}, we have

\begin{Definition}[Relative Barcode Entropy, II]
  \label{def:hbr-rel12}
  The \emph{barcode entropy of $\varphi$ relative to $(L,L')$} is
  $$
  \hbr(\varphi;L,L'):=\lim_{\eps\searrow 0} \hbr_\eps (\varphi;
  L,L')\in [0, \infty],
  $$  
  where
  $$
  \hbr_\eps(\varphi; L,L'):=\limsup_{k\to \infty}\frac{\log^+
    b_\eps\big(L,L^k\big)}{k}\textrm{ with }  L^k:=\varphi^k(L').
  $$
 \end{Definition}
 Here, as in Definition \ref{def:hbr-rel1}, $\hbr_\eps(\varphi; L,L')$
 is increasing as $\eps\searrow 0$, and hence the limit exists,
 although \emph{a priori} it could be infinite.

\begin{Remark}
  The key issue we have to deal with in these definitions is that the
  difference set $\CalD\big(L,L^k\big)$ can be dense, where as above
  $L^k:=\varphi^k(L')$. When the closure $\bar{\CalD}\big(L,L^k\big)$
  is nowhere dense for all $k$, a simpler approach is
  available. Namely, then we can require $\eps\searrow 0$ not to be in
  the union of $\bar{\CalD}\big(L,L^k\big)$ for $k\in\N$ and for each
  $k$ set $b_\eps\big(L,L^k\big):= b_\eps\big(L,\tilde{L}_k\big)$,
  where $\tilde{L}_k$ is a $C^\infty$-small perturbation of $L^k$, as
  in \eqref{eq:b-eps2}.  The resulting definition of the barcode
  entropy would be literally equivalent to the more general one given
  above.
\end{Remark}

\begin{Remark}[Exact Lagrangians, II]
  \label{rmk:filtr-other2}
  Continuing Remark \ref{rmk:filtr-other}, we note that these
  definitions and constructions extend word-for-word to the case where
  $L$ and $L'$ are exact Lagrangian submanifolds.
\end{Remark}

\subsection{Basic properties}
\label{sec:prop}
In this section we list, for the sake of completeness, some basic
properties of barcode entropy.

\begin{Proposition}[Properties of Barcode Entropy] In the notation and
  conventions from Sections \ref{sec:prelim} and \ref{sec:def2}, we
  have the following:
  \label{prop:prop}

  \begin{itemize}

  \item[\reflb{BE1}{\rm{(i)}}] For every $k\in\N$, we have
    $\hbr(\varphi^k; L,L') \leq k\hbr(\varphi; L,L')$. In particular,
    $\hbr(\varphi^k)\leq k\hbr(\varphi)$.

  \item[\reflb{BE2}{\rm{(ii)}}] Assume that the products
    $L_0\times L_1$ and $L'_0\times L'_1$ in $M_0\times M_1$ are
    monotone. Then for Hamiltonian diffeomorphisms
    $\varphi_0\colon M_0\to M_0$ and $\varphi_1\colon M_1\to M_1$, we
    have
  $$
  \hbr(\varphi_0\times\varphi_1; L_0\times L_1, L'_0\times L'_1)\leq
  \hbr(\varphi_0; L_0, L'_0)+\hbr(\varphi_1; L_1, L'_1).
  $$
  In particular,
  $\hbr(\varphi_0\times\varphi_1)\leq
  \hbr(\varphi_0)+\hbr(\varphi_1)$.

\item[\reflb{BE3}{\rm{(iii)}}] We have
  $\hbr(\varphi; L,L')=\hbr(\varphi^{-1};L',L)$ and, in
  particular, $\hbr(\varphi^{-1})=\hbr(\varphi)$.

  \item[\reflb{BE4}{\rm{(iv)}}] For any symplectomorphism
    $\psi\colon M\to M$,
  \[
    \hbr(\varphi;L,L')
    = \hbr\big(\psi\varphi\psi^{-1}; \psi(L),
    \psi(L')\big) \text{ and } \hbr(\varphi)
    =\hbr(\psi\varphi\psi^{-1}).
  \]
  As a consequence, for every $k\in\N$,
  \[
    \hbr(\varphi; L,L')=\hbr\big(\varphi;
    \varphi^k(L),\varphi^k(L')\big).
  \]
 
\item[\reflb{BE5}{\rm{(v)}}] For a fixed Hamiltonian diffeomorphism
  $\varphi$, the barcode entropy $\hbar(\varphi;L,L')$ is lower
  semicontinuous in the pair $(L, L')$ with respect to the Hofer
  distance. In particular, $\hbar(\varphi;L)$ is lower semicontinuous
  in $L$.

\end{itemize}
\end{Proposition}

\begin{proof}
  Assertion \ref{BE1} is a direct consequence of the definition. To
  prove \ref{BE2} recall that the Floer complex of the pair
  $(L_0\times L_1,L'_0\times L'_1)$ is the tensor product of the Floer
  complexes of $(L_0, L'_0)$ and $(L_1,L'_1)$ over $\Lambda$; see,
  e.g., \cite[Sec.\ 2.6]{HLS} and \cite{Li}. Then a singular value
  decomposition for the product is obtained by taking the ``product''
  of singular value decompositions of the factors in a self-evident
  way. As a consequence, every pair of bars $\beta_0\in \CB(L_0,L'_0)$
  and $\beta_1\in \CB(L_1,L'_1)$ gives rise to two bars of length
  $\min\{\beta_0,\beta_1\}$ in $\CB(L_0\times L_1,L'_0\times L'_1)$
  when both bars are finite. If one or both bars in a pair are
  infinite, the pair gives rise to one bar of length
  $\min\{\beta_0,\beta_1\}$.  Therefore,
    $$ b_\eps(L_0\times L_1,L'_0\times L'_1)\leq 2 b_\eps(L_0,L'_0)\cdot
    b_\eps(L_1,L'_1),
    $$
    which proves \ref{BE2}.

    Assertion \ref{BE3} follows from \eqref{eq:inv} and the Poincar\'e
    duality, \eqref{eq:PD}. The first identity in \ref{BE4} also
    follows from \eqref{eq:inv}. The second identity is a consequence
    of the fact that when $\varphi$ is non-degenerate $\varphi$ and
    $\psi\varphi\psi^{-1}$ have isomorphic Floer complexes and hence
    the same barcode. By continuity,
    $\CB(\varphi)=\CB(\psi\varphi\psi^{-1})$ even when $\varphi$ is
    degenerate.  The last one follows from the first identity by
    setting $\psi=\varphi^k$ and using the fact that $\varphi$
    commutes with $\varphi^k$.

    To prove \ref{BE5}, it suffices to show that
  \begin{equation}
    \label{eq:LSC-pair}
  \hbar(\varphi; \tL,\tL')\geq \hbar_{4\delta}(\varphi; L,L'),
\end{equation}
whenever $d_{\hn}(L,\tilde{L})<\delta$ and
$d_{\hn}(L',\tilde{L}')<\delta$. Thus assume that $\tL'=\psi(L')$,
where $\|\psi\|_{\hn}<\delta$. Then
$d_{\hn}\big(\varphi^k,\varphi^k\psi\big)<\delta$, and setting
$\tL^k=\varphi^k(\tL')$ and $L^k=\varphi^k(L')$ as above, we have
$d_{\hn}\big(L^k,\tL^k\big)<\delta$. Therefore, by
\eqref{eq:insensitive1},
  $$
  \hbar_\eps(\varphi; \tL,\tL')\geq \hbar_{\eps+4\delta}(\varphi;
  L,L').
  $$
  Passing to the limit as $\eps\searrow 0$, we obtain
  \eqref{eq:LSC-pair}.
\end{proof}

\begin{Remark}
  By analogy with topological entropy, we would expect \ref{BE1} and
  \ref{BE2} to actually be equalities. Moreover, by Theorem
  \ref{thm:C}, $\hbr(\varphi^k)=k\hbr(\varphi)$ when $M$ is a
  surface. Furthermore, in \ref{BE5} whenever $L$ is wide one can
  replace the Hofer norm by the $\gamma$-norm by the results from
  \cite{KS}.  Also recall that, as was shown in \cite{AM21},
  $\htop(\varphi)$ is lower semicontinuous in $\varphi$ with respect
  to the Hofer metric when $\dim M=2$. Hence, $\hbar(\varphi)$ is also
  Hofer lower semicontinuous for Hamiltonian diffeomorphisms of
  surfaces by Theorem \ref{thm:C}. This observation leads to the
  question/conjecture if/that this is also true in all dimensions.
\end{Remark}

We also note that $\hbr(\varphi;L,L')$, with $\varphi$ fixed, is quite
sensitive to deformations of $L$ and $L'$ even by a Hamiltonian
isotopy; cf.\ Example \ref{exm:moving-L}.

\section{From barcode entropy to topological entropy}
\label{sec:pf-A}
Our goal in this section is to prove Theorem \ref{thm:A} and further
explore some of its consequences. We will establish a slightly more
general result.

\subsection{Generalization of Theorem \ref{thm:A}}
\label{sec:generalization}
With the notation from Sections \ref{sec:def} and
  \ref{sec:def2}, we have the following result generalizing Theorem
\ref{thm:A} to pairs of Lagrangian submanifolds.

\begin{Theorem}
  \label{thm:A2}
  Let $L_0$ and $L_1$ be closed Lagrangian submanifolds of a
  symplectic manifold $M$ and let $\varphi\colon M\to M$ be a
  compactly supported Hamiltonian $C^\infty$-diffeomorphism. Assume
  that $L_0$ and $L_1$ are monotone, Hamiltonian isotopic and
  $N_{L_0}\geq 2.$ Then
\begin{equation}
  \label{eq:hbr-leq-htop}
   \hbr(\varphi; L_0,L_1)\leq \htop(\varphi).
 \end{equation}
\end{Theorem}
Taking $L=L_0=L_1$, we obtain Theorem \ref{thm:A}. We emphasize that
here $M$ need not be compact, but then it has to have sufficiently
``tame'' structure at infinity (e.g., convex) so that the Gromov
compactness theorem holds and $\varphi$ is required to be compactly
supported; see, e.g., \cite{McDS}. Note also that, as a
consequence of Theorem \ref{thm:A2}, $\hbr(\varphi; L_0,L_1)<\infty$
which is otherwise not obvious.

\begin{Remark}[Exact Lagrangians, III]
  \label{rmk:filtr-other3}
  Theorem \ref{thm:A2} holds in some other situations.  For instance,
  one of them is when $L_0$ and $L_1$ are closed exact Lagrangian
  submanifolds in an exact convex symplectic manifold; cf.\ Remarks
  \ref{rmk:filtr-other} and \ref{rmk:filtr-other2}.
\end{Remark}

\subsection{Proof of Theorem \ref{thm:A2}}
\label{sec:pf-A2}
We break down the proof into three subsections. The first two of them
-- Sections \ref{sec:Crofton} and \ref{sec:ELT} -- focus on the
machinery of Lagrangian tomographs which the proof relies on; the
actual proof is given in Section \ref{sec:strat}.  Throughout the
argument we have an auxiliary Riemannian metric on $M$ fixed.

\subsubsection{Lagrangian tomographs and Crofton's inequality}
\label{sec:Crofton}
The notion of Lagrangian tomograph and a variant of Crofton's
inequality, originating in integral geometry, are the key tools used
in the proof of the theorem.  The framework described in this section
is essentially contained in \cite{Ar1, Ar2} in the setting very close
to ours, and we include the proofs only for the sake of completeness.
(See also \cite{BL, CGG:Vol, Se} for other applications in the context
of symplectic dynamics.) Furthermore, our setting is similar to double
fibrations utilized in integral geometry for generalizing and proving
Crofton's formula; see \cite{APF98, APF07, GS}.

Let $L$ be a closed manifold, $B$ a compact manifold possibly with
boundary and $ds$ a smooth measure on $B$. In the situation we are
interested in, $B$ is the closed $d$-dimensional ball $B^d$ and $ds$
is the Lebesgue measure, and $L$ will be the Lagrangian submanifold
$L_0$.  Denote by $\pi\colon E=B\times L\to B$ the projection to the
first factor.  We denote a point in $B$ by $s$.  Furthermore, let
$$
\Psi\colon E \to M
$$ 
be a submersion onto its image where $M$ is a Riemannian manifold (see
Remark \ref{rmk:submersion}). This manifold need not be compact but,
of course, $\Psi(E)\subset M$ is, since $E$ is compact. We require
$\Psi_s:=\Psi|_{s\times L}$ to be an embedding for all $s$, and hence
$L_s:=\Psi_s(L)$ is a smooth closed submanifold of~$M$.

In the spirit of integral geometry we will refer to $\Psi$ as a
\emph{tomograph} and call $L_0$ the \emph{core} of the tomograph. We
say that $\Psi$ is a \emph{Lagrangian tomograph} if all submanifolds
$L_s$ are Lagrangian and Hamiltonian isotopic to each other.

Finally, let $L'$ be a closed submanifold of $M$ with
$$
\codim L'=\dim L.
$$
Since $\Psi$ is a submersion, $\Psi_s\pitchfork L'$ for almost all
$s\in B$. Hence, 
$$
N(s):=|L_s\cap L'|
$$
is a locally constant function on the complement to a zero measure
closed subset of $B$. As a consequence, $N$ is an integrable function
on $B$.

\begin{Lemma}[Crofton's inequality]
  \label{lemma:Crofton}
  We have
  \begin{equation}
    \label{eq:Crofton}
  \int_B N(s)\,ds\leq\const \cdot \vol(L'),
\end{equation}
where the constant depends only on $ds$,
$\Psi$ and the metric on $M$, but not on $L'$.
\end{Lemma}

\begin{Remark}
\label{rmk:submersion}
Perhaps a clarification is due on how the submersion condition is to
be interpreted at $\p B$. A way, which is sufficient and convenient
for our purposes, is to assume that $\Psi$ is defined on a slightly
larger space $B'\times L$ (or just $E'\supset E$) where $B'$ is an
open enlargement of $B$ (or $E'$ is an open enlargement of $E$).
\end{Remark}

Lemma \ref{lemma:Crofton} is proved in \cite{Ar1,Ar2}.  The argument
is simple and short, and we include it below for the sake of
completeness. This general framework and Lemma \ref{lemma:Crofton} are
also very much in the spirit of the Gelfand transform in integral
geometry and various versions of Crofton's formula; see, e.g.,
\cite{APF98,APF07} and references therein. The key difference is that
here and in \cite{Ar1,Ar2} $\Psi$ is required to be only a submersion,
not a fibration. Furthermore, we note that in this generality, in
contrast with the actual Crofton formula, one cannot expect an
inequality going in the direction opposite of
\eqref{eq:Crofton}. Indeed, the graph of a smooth map or even of a
diffeomorphism between two closed manifolds can have arbitrarily large
volume but it intersects every vertical slice at only one point.

\begin{proof}[Proof of Lemma \ref{lemma:Crofton}]
  Set $\Sigma=\Psi^{-1}(L')$. This is a smooth submanifold of $E$ and
  $$
  \codim \Sigma = \codim L' = \dim L, \textrm{ i.e., } \dim\Sigma =
  \dim B.
  $$
  By construction,
  \begin{equation}
    \label{eq:L-Sigma}
  |(s\times L)\cap \Sigma|=|L_s\cap L'|=N(s).
  \end{equation}

  In the proof it will be convenient to equip $E$ with two different
  auxiliary metrics: the first metric adapted to $\pi$ and the second
  metric to $\Psi$.

  We begin by fixing some metrics on $B$ and $L$, and assuming first
  that $E=B\times L$ carries the product metric and $ds$ is the
  Riemannian volume form or, to be more precise, the volume
  density. (It would be sufficient to require $D\pi$ to be an isometry
  on the normals to the fibers.)  Then, by \eqref{eq:L-Sigma},
  $$
  \int_B N(s)\,ds=\int_\Sigma\pi^*ds\leq \vol(\Sigma).
  $$
  Here, $\pi^*ds$ is the pull-back measure or the pull-back density,
  but not the pull-back differential form. The last inequality is a
  consequence of the fact that $D\pi_x\colon T_xE\to T_{\pi(x)}B$,
  $x\in E$, is an orthogonal projection along the fiber, and hence,
  when restricted to $T_x\Sigma$, it can only decrease the
  $\dim B$-dimensional volume. As a consequence, for an arbitrary
  metric on $E$ and an arbitrary smooth measure $ds$ on $B$, we have
  \begin{equation}
    \label{eq:N-Sigma}
  \int_B N(s)\,ds\leq \const\cdot\vol(\Sigma).
\end{equation}

Next, let us equip $E$ with a metric such that the restriction of
$D\Psi$ to the normals to the fibers of $\Psi$ (i.e., the inverse
images $\Psi^{-1}(y)$, $y\in M$) is an isometry. Then, by Fubini's
theorem or, more specifically, the coarea formula (see \cite[Sect.\
13.4.3]{BZ88}), we have
\begin{equation}
  \label{eq:const}
  \vol(\Sigma) =\int_{L'}\vol\big(\Psi^{-1}(y)\big)\,dy|_{L'}
  \leq\max_{y\in\Psi(E)}\vol\big(\Psi^{-1}(y)\big)\cdot
  \vol(L'),
\end{equation}
where in the first equality $dy|_{L'}$ stands for the induced volume
form on $L'$. Thus
  \begin{equation}
    \label{eq:Sigma-L}
    \vol(\Sigma)\leq\const\cdot \vol(L').
  \end{equation}
  For an arbitrary metric on $E$, \eqref{eq:Sigma-L} still holds,
  albeit with a different constant. Combining \eqref{eq:N-Sigma} and
  \eqref{eq:Sigma-L}, we obtain \eqref{eq:Crofton}.
\end{proof}

\begin{Remark}
  \label{rmk:N-finite}
  Note that by \eqref{eq:const} and since the constant in
  \eqref{eq:N-Sigma} is independent of $\Psi$, the constant in
  \eqref{eq:Crofton} is continuous in $\Psi$ in the $C^1$-topology. In
  our application $L'=L^k$, i.e., it ranges through a countable
  collection of submanifolds of $M$. Then, by $C^1$-perturbing $\Psi$
  slightly, we can ensure that $N(s)$ is finite for all $s$. (This
  fact is inessential for our purposes, and we omit a proof.) Hence,
  the functions $N_k(s)$ from Section \ref{sec:strat} can also be made
  finite for all $s$.
\end{Remark}

\subsubsection{Existence of Lagrangian tomographs}
\label{sec:ELT}

As the second step of the proof we establish in this section the
existence of Lagrangian tomographs. Let $L=L_0$ be a closed Lagrangian
submanifold of $M$.

\begin{Lemma}
  A Lagrangian tomograph with core $L$ and $\dim B=d$ exists if and
  only if $L$ admits an immersion into $\R^d$.
\end{Lemma}

\begin{proof}
  By the Weinstein tubular neighborhood theorem it is sufficient to
  prove the lemma when $L$ is the zero section in $M=T^*L$.
  
Let $\iota\colon L \to \R^d$ be an immersion, and 
$$
f_s=s_1 g_1+\ldots + s_d g_d,
$$
where $s=(s_1,\ldots,s_d)\in \R^d$ are the coordinate functions on
$\R^d$ and $g_i:=s_i\circ \iota$ are the restrictions of the
coordinate functions to $L$. Let us now require $s$ to be in a ball
$B\subset \R^d$ centered at the origin. Then, setting
$\Psi_s(x)=df_s(x)$, $x\in L$, we obtain a map
$\Psi\colon B\times L\to T^*L$. It is easy to see that the condition
that $\iota$ is an immersion is equivalent to that $\Psi$ is a desired
tomograph.

Indeed, the immersion condition is equivalent to that
$dg_1,\ldots, dg_d$ generate $T_x^*L$ at every point $x\in L$.  In
other words, the map
  $$
  D\Psi_{(0,x)}\colon T_0B\oplus T_xL\to T_{(0,x)}T^*L
  = T^*_xL\oplus T_xL
  $$
  is onto. Therefore, $\Psi$ is a submersion when $B$ is sufficiently
  small if and only if $(g_1,\ldots,g_d)\colon L\to \R^d$ is an
  immersion. (This observation is already contained in, e.g.,
  \cite{GuSt}.)  It is clear from the construction that the
  submanifolds $L_s$ are embedded.

  Conversely, assume that $\Psi\colon L\times B\to T^*L$ is a
  tomograph with $\Psi_0=\id$. Then the linearization
  $A=\p \Psi_s/\p s$ is a linear map from $T_0B$ to the space of exact
  sections of $T^*L$. Pick some functions $\{g_\ell\}$ so that
  $dg_\ell=A(e_\ell)$ where $e_1,\ldots, e_d$ is a basis in
  $T_0B$. Then $\iota=(g_1,\ldots,g_d)$ is an immersion $L\to\R^d$;
  cf.\ \cite{GuSt}. This completes the proof of the lemma.
  \end{proof}

\subsubsection{From barcodes to Lagrangian volume to entropy}
\label{sec:strat}
Without loss of generality, we may require that
$\hbr(\varphi; L_0,L_1)>0$ -- otherwise there is nothing to prove.  We
denote the Riemannian volume of $L^k:=\varphi^k(L_1)$ by $\vol(L^k)$.

Fix $\eps>0$ and $\alpha < \hbr_{2\eps}(\varphi; L_0,L_1)$. Then
\begin{equation}
  \label{eq:beps-alpha}
  b_{2\eps}\big(L_0,L^{k_i}\big)\geq
  \const \cdot 2^{\alpha k_i} 
\end{equation}
for some sequence $k_i\to\infty$. (Here and in what follows, the value
of the constant $\const$ can change from one formula to another and
even in different parts of the same formula.) By Yomdin's theorem (see
\cite{Yo} and also the survey \cite{Gr}), it is sufficient to show
that
\begin{equation}
  \label{eq:vol}
\vol(L^{k_i})\geq \const \cdot 2^{\alpha k_i}
\end{equation}
with the constant independent of $i$, but possibly depending on
  $\alpha$ and $\eps$.  Indeed, then $\alpha\leq \htop(\varphi)$.
Passing to the limit as $\eps\to 0$ and
$\alpha\to \hbr(\varphi; L_0,L_1)$, we obtain~\eqref{eq:hbr-leq-htop}.

To prove \eqref{eq:vol}, pick a Lagrangian tomograph with core $L_0$.
Thus we have a family of Lagrangian submanifolds $L_s$ smoothly
parametrized by the closed $d$-dimensional ball $B^d=B^d(r)$ of radius
$r>0$ (for some large $d$). By shrinking $B$ if necessary, we can
ensure that these submanifolds have the following properties:
\begin{itemize}
\item[\reflb{ELT1}{(i)}] The Lagrangian submanifolds $L_s$ are
  Hamiltonian isotopic to $L_0$ and the Hofer distance between $L_0$
  and $L_s$ is small:
  \begin{equation}
    \label{eq:H-close}
  d_{\hn}(L_0,L_s)<\eps/2
\end{equation}
(In fact, $L_s$ can even be taken $C^\infty$-close to $L_0$.)

\item[\reflb{ELT2}{(ii)}] The Lagrangian submanifold $L_s$ is
  transverse to $L^k$ for all $k$ and almost all $s\in B^d$. Let
  $$
  N_k(s):=|L_s\cap L^k|
  $$
  be the number of intersections of $L_s$ and $L^k$. Then $N_k(s)$ is
  a measurable function, finite for almost all $s$. (In fact, we can
  even have $N_k(s)$ finite for all $s\in B^d$; see Remark
  \ref{rmk:N-finite}.) Furthermore, by Crofton's inequality (Lemma
  \ref{lemma:Crofton}), we have
\begin{equation}
  \label{eq:vol-int}
  \int_{B^d} N_k(s)\, ds\leq \const\cdot \vol (L^k),
\end{equation}
where $ds$ is the standard Lebesgue measure on $B^d$.
\end{itemize}
Note that all conditions in \ref{ELT1} and \ref{ELT2}, other than
\eqref{eq:H-close}, are satisfied automatically; see Section
\ref{sec:Crofton}. As we have already pointed out, to guarantee
\eqref{eq:H-close} we can simply shrink $B$.

Then, whenever $L_s\pitchfork L^k$, we have
$$
  N_k(s) \geq b_\eps\big(L_s,L^k\big)
         \geq b_{2\eps}\big(L_0,L^k\big).
$$
This is an immediate consequence of \eqref{eq:intersections}
with $L$ replaced by $L^k$, $L'$ replaced by $L_0$ and $L''$
replaced by $L_s$.  Hence, by \eqref{eq:beps-alpha},
$$
N_{k_i}(s)\geq \const \cdot 2^{\alpha k_i}
$$
for almost all $s\in B^d$, and \eqref{eq:vol} follows from
\eqref{eq:vol-int}, which finishes the proof of Theorem
  \ref{thm:A2}.

  We emphasize that this argument does not require $L_0$ and $L^k$ to
  be transverse, but only that $L_s\pitchfork L^k$ which holds
  automatically for almost all $s$ regardless of whether
  $L_0\pitchfork L^k$ or not.

\subsection{Entropy and the graph volume growth}
\label{sec:vol}
A consequence of the proof of Theorem \ref{thm:A2} and Theorem
\ref{thm:C} is a relation between the barcode and topological entropy
of $\varphi$ and the exponential growth rate of the volume of the
graph of $\varphi^k$. To be more precise, assume that $M$ is compact
and fix a Riemannian metric on $M$. Denote by
$\Gamma_k\subset M\times M$ the graph of $\varphi^k$ and by
$\vol(\Gamma_k)$ its volume. Set
$$
\hvol(\varphi)=\limsup_{k\to\infty}\frac{\log^{+}\vol(\Gamma_k)}{k}.
$$

\begin{Corollary}
  \label{cor:vol1}
  Let $\varphi$ be a Hamiltonian $C^\infty$-diffeomorphism of a
  compact monotone symplectic manifold. Then
$$
\hbar(\varphi)\leq \hvol(\varphi)\leq \htop(\varphi).
$$
\end{Corollary}
Here the first inequality is an immediate consequence of the proof of
Theorem \ref{thm:A2}, and the second one follows from Yomdin's
theorem, \cite{Yo}, and holds for general
$C^\infty$-diffeomorphisms. Combining these inequalities with Theorem
\ref{thm:C}, we obtain

\begin{Corollary}
  \label{cor:vol2}
  Let $\varphi$ be a Hamiltonian $C^\infty$-diffeomorphism of a closed
  surface. Then
$$
\hvol(\varphi)= \htop(\varphi).
$$
\end{Corollary}
Surprisingly, this equality appears to be new. There are however
similar results in the holomorphic setting; see \cite{Gr:hol}
and the comments therein by S. Cantat for further references.

\section{From horseshoes to barcode entropy}
\label{sec:pf-BC}
Our goal in this section is to prove Theorems \ref{thm:B} and
\ref{thm:C}. Throughout the proofs, $\varphi^t:=\varphi^t_H$ will
stand for the time-dependent Hamiltonian flow of
$H\colon S^1\times M\to \R$, $S^1=\R/\Z$, on a (monotone) symplectic
manifold $M$ and $\varphi:=\varphi_H^1$ will be the Hamiltonian
diffeomorphism generated by $H$.

\subsection{Crossing energy, the proof of Theorem \ref{thm:B} and the
  $\gamma$-norm}
\label{sec:pf-B+energy}
In this section we prove Theorem \ref{thm:B} and also briefly touch
upon an application of the proof to establishing a lower bound on the
$\gamma$-norm of the iterates in the presence of a hyperbolic set. We
break down the proof of the theorem in a few simple steps, some of
which might be of independent interest.

\subsubsection{Generalities and terminology}
\label{sec:general}
The iterate Hamiltonian diffeomorphism $\varphi^k$ is the time-$k$ map
in the Hamiltonian isotopy $\varphi_H^t$ generated by $H$. In what
follows, when working with the Floer equation for this iterate, it is
convenient to denote the Hamiltonian by $H^{\sharp k}$ and refer to
it, somewhat abusing terminology, as the \emph{iterated Hamiltonian};
cf.\ \cite{GG:hyperbolic,GG:PR}. We emphasize that $H^{\sharp k}$ is
the same Hamiltonian as $H$, but now viewed as $k$-periodic in
time. Likewise, the solutions of the Floer equation, \eqref{eq:Floer},
are allowed to be $k$-periodic in time rather than $1$-periodic or,
more generally, defined on a closed domain
$\Sigma\subset \R\times S^1_k$, where $S^1_k=\R/k\Z$, rather than a
domain in $\R\times S^1$. There are, of course, other Hamiltonians,
with easily adjustable period, generating $\varphi^k$ and giving rise
to the same filtered Floer complex, but this is a natural and
convenient choice from the dynamics perspective. Moreover, this choice
becomes essential for the proof of Theorem \ref{thm:cross-energy}; see
\cite{GG:PR}.

Thus, consider solutions $u\colon \Sigma\to M$ of the Floer
equation
\begin{equation}
  \label{eq:Floer}
 J\p_s u=\p_t u-J\nabla H^{\nat k}
\end{equation}
for the iterated Hamiltonian $H^{\nat k}\colon S^1_k\times M\to \R$
with $S^1_k=\R/k\Z$, where $\Sigma\subset \R\times S^1_k$ is a closed
domain, i.e., a closed subset with non-empty interior and $J$ is a
background $k$-periodic in time almost-complex structure. By
definition, the energy of $u$ is
$$
 E(u)=\int_\Sigma \|\p_s u\|^2 \, ds dt.
$$
where $\|\cdot\|$ stands for the norm with respect to
$\left<\cdot\, ,\cdot\right>=\omega(\cdot,J\cdot)$, and hence
$\|\cdot\|$ depends on $J$.  Recall that when $\Sigma=\R\times S^1_k$
and $u$ is asymptotic to $k$-periodic orbits $x$ at $-\infty$ and $y$
at $\infty$, we have
$$
E(u)=\CA(x)-\CA(y),
$$
i.e., $E(u)$ is the action difference between $x$ and $y$. Here we
treat $x$ and $y$ as capped $k$-periodic orbits of $H$ with the
capping of $y$ obtained by ``attaching'' $u$ to the capping of $x$;
see Section \ref{sec:Floer-complex}.

Throughout the proof, it will be convenient to work with the extended
phase space $\tilde{M}=S^1\times M$ with $S^1=\R/\Z$. The
time-dependent flow $\varphi^t$ lifts as the genuine flow
$\tilde{\varphi}^t$ on $\tilde{M}$ given by
$$
 \tilde{\varphi}^t(\theta,p)=\big(\theta+t,\varphi^t(p)\big)
$$
generated by the vector field $\p_\theta +X_H$, where in the first
term $t$ is viewed as an element of $S^1=\R/\Z$. Likewise, any map
$z\colon \R\to M$ or $z\colon S^1_k\to M$ lifts to the map
$\tilde{z}(t)=(t, z(t))$, and in a similar vein a solution $u$ of the
Floer equation lifts to a map $\tilde{u}\colon \Sigma\to
\tilde{M}$. If $u$ is asymptotic to $x$ and $y$, the lift $\tilde{u}$
is asymptotic to $\tx$ and $\ty$ in the natural sense. In what
follows, a lift from $M$ to $\tilde{M}$ will always be indicated by
the tilde and we will identify $M$ with $\{0\}\times M$.

A loop $x\colon S^1_k\to M$ is a $k$-periodic orbit of $\varphi^t$ if
and only if its lift $\tx$ is a $k$-periodic orbit of
$\tilde{\varphi}^t$ and if and only if the sequence
$\hat{x}=\{x_i:=x(i)\mid i\in\Z_k\subset S^1_k\}$ formed by the
intersections of $\tx$ with the cross-section $M$ is a $k$-periodic
orbit of $\varphi$.

\subsubsection{Crossing energy}
\label{sec:cross-energy}
Next let us recall the crossing energy theorem, \cite[Thm.\
6.1]{GG:PR} (see also \cite{GG:hyperbolic}), which is crucial to the
proof.  Let $K\subset M$ be a compact invariant set of a Hamiltonian
diffeomorphism $\varphi$ of $M$. Recall that $K$ is said to be
\emph{locally maximal} or isolated (as an invariant set) or basic if
there exists a neighborhood $U\supset K$ such that for no initial
condition $p\in U\setminus K$ the orbit through $p$ is contained in
$U$, i.e., there exists $k\in\Z$, possibly depending on $p$, such that
$\varphi^k(p)\not\in U$. The neighborhood $U$ is called an isolating
neighborhood of $K$. Then any neighborhood of $K$ contained in $U$ is
also isolating, and hence such neighborhoods can be made arbitrarily
small. In other words, whenever $U\supset V\supset K$ and $U$ is an
isolating neighborhood and $V$ is open, $V$ is also an isolating
neighborhood. For a flow, a locally maximal set is defined in a
similar fashion.

The set $K$ naturally lifts to an invariant set
$\tilde{K}\subset \tilde{M}$ of the flow $\tilde{\varphi}^t$, which is
the union of the integral curves through $K=\{0\}\times K$. (Since $K$
is invariant it suffices to take only $t\in [0,1]$.) The set $K$ is
locally maximal for $\varphi$ if and only if $\tilde{K}$ is locally
maximal for the flow.

As in Section \ref{sec:general}, let $u\colon \Sigma\to M$ be a
solution of the Floer equation, \eqref{eq:Floer}, where
$\Sigma\subset \R\times S^1_k$ is a closed domain.  We say that $u$ is
\emph{asymptotic} to $K$ at $\infty$ (or at $-\infty$) if for any
neighborhood $\tilde{U}$ of $\tilde{K}$ there is a half-cylinder
$[s_{\tilde{U}},\,\infty)\times S^1_k$ (or
$(-\infty,\,s_{\tilde{U}}]\times S^1_k$) in $\Sigma$ which is mapped
into $\tilde{U}$ by $\tilde{u}$. For instance, $u$ is asymptotic to
$K$ whenever $u(s,\cdot)$ uniformly converges as $s\to \infty$ or
$s\to -\infty$ to a $k$-periodic orbit $x$ with $x(0)\in K$ (but not
necessarily with $x(t)\in K$ for all $t\in S^1_k$). In this case,
abusing terminology, we will also say that $u$ is asymptotic to the
$k$-periodic orbit $\hat{x}:=\{x_i:=x(i)\mid i\in\Z_k\}$ of
$\varphi$. We emphasize that here the domain $\Sigma$ of $u$ need not
be a cylinder, although to be asymptotic to $K$ it must contain a
half-cylinder.

Finally, fix a (sufficiently small) isolating neighborhood $\tilde{U}$
of $\tilde{K}$. Set
$\p \tilde{U}:=\mathrm{closure}(\tilde{U})\setminus \tilde{U}$.

\begin{Theorem}[Crossing Energy Theorem, Thm.\ 6.1 in \cite{GG:PR}]
  \label{thm:cross-energy}
  Fix a 1-periodic in time almost complex structure $J$ on $M$.  Let
  $J'$ be a $k$-periodic in time almost complex structure on $M$ which
  is sufficiently $C^\infty$-close to $J$, depending on $k$, uniformly
  on $U$. Furthermore, let $u\colon \Sigma\to M$, where
  $\Sigma\subset \R\times S^1_k$, be a solution of the Floer equation
  for $J'$ and $H^{\sharp k}$, asymptotic to $K$ as $s\to\infty$ or
  $s\to-\infty$, and such that
 \begin{itemize}

 \item[\reflb{CEa}{(a)}] either $\p\Sigma\neq \emptyset$ and
   $\tilde{u}(\p\Sigma)\subset \p \tilde{U}$

 \item[\reflb{CEb}{(b)}] or $\Sigma=\R\times S^1_k$ and
   $\tilde{u}(\Sigma)\not\subset \tilde{U}$.
 \end{itemize}
 Then there exists a constant $c_\infty>0$, independent of $k$, $J'$,
 $u$ and $\Sigma$ such that
 \begin{equation}
   \label{eq:cross-energy}
   E(u)>c_\infty .
 \end{equation}

\end{Theorem}

Here we are mainly interested in Case \ref{CEb} of the theorem. It is
easy to see that this case is a consequence of the more general Case
\ref{CEa} which was actually established in \cite{GG:PR}. However,
Case \ref{CEb} can also be proved directly by an argument which is
considerably simpler than the original proof therein; cf.\ Remark
\ref{rmk:cross-energy}.

\begin{Remark}
  It is worth keeping in mind that the lower bound $c_\infty$ depends
  on the choice of an isolating neighborhood $\tU$ of $\tK$: a smaller
  neighborhood might necessitate a smaller lower bound. The threshold
  on how close $J'$ and $J$ need to be for \eqref{eq:cross-energy} to
  hold depends on $k$. Finally, as readily follows from the proof, the
  lower bound $c_\infty$ can also be chosen to be stable with respect
  to $C^\infty$-small perturbations of $H^{\sharp k}$, i.e., so that
  \eqref{eq:cross-energy} holds for solutions of the Floer equation
  for all $k$-periodic Hamiltonians $C^\infty$-close to
  $H^{\sharp k}$.
\end{Remark}

\subsubsection{Energy bound for Floer trajectories asymptotic to $K$.}
\label{sec:energy-bound}
Consider Floer trajectories $u\colon \Sigma=\R\times S^1_k\to M$ for
some $k$-periodic almost complex structure $J'$ sufficiently close to
a fixed 1-periodic almost complex structure $J$ as in Theorem
\ref{thm:cross-energy} and asymptotic to $k$-periodic orbits $x$ and
$y$ of $\varphi^t$ with $x(0)\in K$.  (It does not matter if $u$ is a
asymptotic to $x$ at $\infty$ or $-\infty$, and whether $y(0)$ is in
$K$ or not.) The key to the proof of Theorem \ref{thm:B} is the
following result which is an easy consequence of Theorem
\ref{thm:cross-energy} and the Anosov Closing Lemma, \cite[Thm.\
6.4.15]{KH}.

\begin{Proposition}
  \label{prop:energy}
  Assume that $K$ is a locally maximal hyperbolic invariant set of
  $\varphi$. Then
  \begin{equation}
    \label{eq:gap}
  E(u)=|\CA(x)-\CA(y)|>\eps_K, \textrm{ unless } E(u)=0,
 \end{equation}
 for some constant $\eps_K>0$, independent of $u$ and $x$ and $y$, and
 also of $k$ and $J'$ as long as $J'$ is sufficiently $C^\infty$-close
 to $J$ in the class of $k$-periodic in time almost complex structures
 on $M$.
  \end{Proposition}

  \begin{proof} Pick an almost complex structure $J'$ which is
    sufficiently close to a 1-periodic almost-complex structure $J$
    and denote by $d$ the distance on $M$ with respect to an arbitrary
    Riemannian metric. We will need the following standard fact:

    \begin{Lemma}
      \label{lemma:pseudo-orbit}
      Let $u\colon \R\times S^1_k\to M$ be a solution of the Floer
      equation, \eqref{eq:Floer}. Assume that $E(u)$ is sufficiently
      small, i.e., $E(u)<e$ with an upper bound $e>0$ depending only
      on $(M, \omega, J)$ and $H$. Then for every $s\in\R$ the set
    $$
    \hat{z}:=\{z_i:=u(s,i)\mid i\in \Z_k\}
    $$
    is a periodic $\eta$-pseudo-orbit of $\varphi$, i.e.,
    \begin{equation}
      \label{eq:pseudo-orbit}
      d\big(\varphi(z_i),z_{i+1}\big)<\eta \textrm{ for all $i\in\Z_k$},
    \end{equation}
    where we can take
    \begin{equation}
      \label{eq:eta-energy}
    \eta= O\big(E(u)^{1/4}\big).
  \end{equation}
\end{Lemma}

The key point of this lemma is that every ``circle'' in a Floer
cylinder $u$ for $H^{\# k}$ is in essense an $\eta$-pseudo-orbit with
$\eta= O\big(E(u)^{1/4}\big)$, provided that $E(u)$ is below a certain
threshold $e>0$ which depends only on $(M, \omega, J)$ and $H$, but
not $u$ or $k$.
  
\begin{proof}
  Recall that when $E(u)$ is sufficiently small (with an upper bound
  $e$ depending on $M$ and $H$ but not $u$ and $k$), we have the
  pointwise upper bound
  \begin{equation}
    \label{eq:ptwise-bnd}
    \| \p_s u\|\leq \const\cdot E(u)^{1/4}=O\big(E(u)^{1/4}\big),
  \end{equation}
  where the constant is again independent of $u$ and $k$ and of $J'$
  when $J'$ is close to $J$; see \cite[Sect.\ 1.5]{Sa} or
  \cite[p.\ 542--543]{GG:LS} or, for a different proof,
  \cite{Br}. (Note that it is essential here that the domain of $u$ is
  the entire cylinder $\R\times S^1_k$.) Now, \eqref{eq:pseudo-orbit}
  and \eqref{eq:eta-energy} follow from the Floer equation,
  \eqref{eq:Floer}, and the triangle inequality.
  \end{proof}
  
  To prove the proposition, we consider two cases depending on the
  location of the orbit~$y$.

  The first case is when $y(0)\not\in K$. Then $\ty$ is not entirely
  contained in any isolating neighborhood of $\tilde{K}$. Applying
  Theorem \ref{thm:cross-energy} (Case b), we obtain \eqref{eq:gap}
  with $\eps_K=c_\infty$.

  The second case is when $y(0)\in K$. Then both $\hat{x}$ and
  $\hat{y}:=\{y_i:=y(i)\mid i\in\Z_k\}$ are $k$-periodic orbits of
  $\varphi$ in $K$. Let us assume that $u$ is asymptotic to $x$ at
  $-\infty$; the other case is handled similarly.  We will show that
  then $x=y$ and $E(u)=0$ when $E(u)$ is below a certain treshold
  which depends only on $M$, $J$, $K$ and $H$, but neither on $u$ nor
  $x$ nor $y$ nor~$k$.

  To this end, fix a sufficiently small isolating neighborhood $U$ of
  $K$. In particular, we may assume that the Anosov Closing Lemma
  applies to $\varphi$ on $U$; see \cite[Thm.\ 6.4.15]{KH}. Then, by
  Theorem \ref{thm:cross-energy}, $\tilde{u}$ is entirely contained in
  $\tilde{U}$. Hence, $\hat{z}$ is contained in $\tilde{U}\cap M=U$
  for all $s\in \R$.  Assume furthermore that $E(u)$ is so small that
  Lemma \ref{lemma:pseudo-orbit} applies and thus
  $\hat{z}=\{u(s,i)\mid i\in\Z_k\}$ is a periodic $\eta$-pseudo-orbit
  in $U$ for every $s$.

  Therefore, by the Anosov Closing Lemma, there exists a true periodic
  orbit $\hat{w}$ in $K$ shadowing $\hat{z}$. Namely, we have
  $d(z_i,w_i)<C\eta$, $i\in\Z_k$, for some constant $C>0$, which
  depends only on $U$ and $\varphi$. By \cite[Cor.\ 6.4.10]{KH},
  $\varphi|_K$ is expansive: there is a constant $\delta>0$ such that
  any two distinct orbits $\{v_i\}$ and $\{v'_i\}$ of $\varphi$ in $K$
  are at least $\delta$ apart, i.e., $d(v_j,v'_j)>\delta$ for some
  $j\in \Z$. It follows that when $E(u)$ and hence $\eta$ are small
  enough (e.g., $2C\eta<\delta$), the orbit $\hat{w}$ is unique and
  depends continuously on $\hat{z}$ and thus on $s\in\R$. Therefore,
  again since $\varphi|_K$ is expansive, $\hat{w}$ is independent of
  $s\in\R$. Clearly, when $s$ is close to $-\infty$, we have
  $\hat{w}=\hat{x}$, and $\hat{w}=\hat{y}$ when $s$ is close to
  $\infty$. Thus, $x=y$, and setting $u(\infty, t)=x(t)=y(t)$, we can
  view $u$ as a $C^0$-map from
  $\T^2=\big(\R\cup\{\infty\}\big)\times S^1_k$ to $M$. This map is
  smooth on the complement to $\{\infty\}\times S^1_k$. Furthermore,
  it is easy to see from \eqref{eq:ptwise-bnd} and Lemma
  \ref{lemma:pseudo-orbit} that for every $s\in \R$ the loop
  $t\mapsto u(s,t)$ is $C^0$-close to the loop $x=y$ pointwise
  uniformly in $s$.  Hence, the loop $s\mapsto u(s,t)$ lies in a small
  neighborhood of $x(t)$. As a consequence, $u$ contracts to $x$ in
  $M$, and hence $E(u)=0$.  Indeed, $E(u)$ is the difference of
  actions of capped periodic orbits. Since $x=y$, this difference is
  the integral of $\omega$ over~$u$. The cycle represented by $u$ is
  homologous to zero, and hence the integral is zero.
\end{proof}
\begin{Remark}
  \label{rmk:cross-energy}
  The proof of Case \ref{CEa} of Theorem \ref{thm:cross-energy} in
  \cite{GG:PR} relies on a variant of the Gromov Compactness Theorem
  from \cite{Fi}. As has already been mentioned, Case \ref{CEb} used
  here follows from Case \ref{CEa}, but it can also be proved directly
  with somewhat simpler tools under the slightly more restrictive
  requirement that, as in Lemma \ref{lemma:pseudo-orbit}, $J'$ is
  sufficiently $C^\infty$-close to $J$ in the class of $k$-periodic in
  time almost complex structures on $M$. Namely, arguing by
  contradiction, assume that $E(u)$ can be arbitrarily small, i.e.,
  there exists a sequence $u_k\colon \R\times S^1_k\to M$, where
  $k=k_i\to \infty$ with $E(u_k)\to 0$ such that the image of $\tu_k$
  is not entirely contained in the closure of $\tU$.  Consider the
  largest half-cylinder in $\R\times S^1_k$ whose image is contained
  in the closure. By Lemma \ref{lemma:pseudo-orbit}, the restriction
  of $u_k$ to the boundary of this cylinder gives rise to an
  $\eta$-pseudo-orbit with $\eta=O\big(E(u_k)\big)$ passing through a
  point close to $\p U$. Thus we obtain longer and longer
  two-directional $\eta$-pseudo-orbits with $\eta\to 0$ passing
  through a point close to $\p U$. Passing to the limit as
  $E(u_k)\to 0$, we obtain an entire orbit of $\varphi$ which is
  contained in the closure of $U$, but not in $K$. The argument is
  spelled out in detail in a very similar context in
  \cite{CGGM:hyperbolic, CGGM:entropy}.
\end{Remark}

\begin{Remark}
  Since this work appeared as a preprint, variants of Proposition
  \ref{prop:energy} for geodesic and Reeb flows were proved in
  \cite{GGM} and in \cite{CGGM:hyperbolic, CGGM:entropy}, leading to
  Reeb analogues to Theorems \ref{thm:B} and \ref{thm:C} and also of
  multiplicity results along the lines of \cite{GG:hyperbolic}. A
  version of Proposition \ref{prop:energy} for Lagrangian
  intersections has been recently proved in \cite{Me:Entropy}.
\end{Remark}

\subsubsection{Proof of Theorem \ref{thm:B} }
\label{sec:pf-B}
Recall that a compact invariant set $K$ of $\varphi$ is said to be
\emph{locally maximal} (or basic) if there exists a neighborhood
$U\supset K$, called an \emph{isolating neighborhood}, such that $K$
is the maximal invariant subset of $U$ or, in other words, $x\in K$
whenever the entire orbit $\{\varphi^k(x)\mid k\in\Z\}$ through $x$ is
contained in $U$ or, equivalently,
$$
K=\bigcap_{k\in\Z}\varphi^k(U).
$$
By \cite[Thm.\ 3.3]{ACW} and the Variational Principle, \cite[Thm.\
4.5.3]{KH}, for every hyperbolic set $K$ there exists a locally
maximal hyperbolic set $K'$ with nearly the same entropy.  (In fact,
the hyperbolicity condition is essential for \cite[Thm.\ 3.3]{ACW} but
not for the Variational Principle.)  To be more precise, for every
$\delta>0$, one can find $K'$ such that
$\htop\big(\varphi|_{K'}\big)\geq \htop
\big(\varphi|_K\big)-\delta$. As a consequence, we can assume the
hyperbolic set $K$ in the theorem to be locally maximal.

Denote by $p(k)$ the number of $k$-periodic points of
$\varphi|_K$. Since $K$ is hyperbolic, by \cite[Thm.\ 18.5.1]{KH}, we
have
\begin{equation}
\label{eq:p-htop}
\htop(\varphi|_K)=\limsup_{k\to\infty} \frac{\log^+ p(k)}{k}.
\end{equation}
Hence, to prove the theorem, it is sufficient to show that 
\begin{equation}
  \label{eq:p-b}
b_\eps(\varphi^k)\geq p(k)/2
\end{equation}
when $\eps>0$ is small. We will use Proposition \ref{prop:energy} to
prove this for $\eps<\eps_K$.

Fix $k\geq 1$ and recall from Section \ref{sec:def2} that the limit
\eqref{eq:b-eps3} in the definition of $b_\eps(\varphi^k)$ is
attained, i.e, there exists non-degenerate, arbitrarily
$C^\infty$-small perturbations $\psi$ of $\varphi^k$ such that
\begin{equation}
  \label{eq:psi-phi}
  b_\eps(\varphi^k)= b_\eps(\psi).
\end{equation}

All $k$-periodic points of $\varphi$ in $K$ are non-degenerate and
hence persist as fixed points of $\psi$ or, equivalently, as
$k$-periodic orbits $S^1_k\to M$ of the Hamiltonian flow
$\psi^t$. Moreover, by Proposition \ref{prop:energy}, $E(u)>\eps_K$
for any Floer trajectory $u$ asymptotic to such an orbit $x$ of $\psi$
when $\psi$ is sufficiently $C^\infty$-close to $\varphi^k$.

Let $\mathcal{K}$ be the collection of fixed points of $\psi$
corresponding to the $k$-periodic points of $\varphi$ in $K$. There
are exactly $p(k)$ of them: $|\mathcal{K}|=p(k)$.  By Proposition
\ref{prop:energy} every such fixed point is $\eps$-isolated in the
sense of Section \ref{sec:isolated} and, by Proposition
\ref{prop:isolated} and \eqref{eq:psi-phi},
$$
b_\eps(\varphi^k)= b_\eps(\psi)\geq |\mathcal{K}|/2=p(k)/2,
$$
which proves \eqref{eq:p-b} and completes the proof of the
theorem. \qed

\subsubsection{Application to the $\gamma$-norm of the iterates}
\label{sec:gamma}
The proof of Theorem \ref{thm:B} has an application to symplectic
topology, relating the behavior of the $\gamma$-norm
$\gamma\big(\varphi^k\big)$ and the Hofer norm $\|\varphi^k\|_{\hn}$
as $k\to\infty$ to the existence of hyperbolic locally maximal
invariant subsets. Namely, recall from \cite[Thm.\ A]{KS} that
$$
\beta_{\max}(\varphi)\leq \gamma(\varphi) \leq \|\varphi\|_{\hn},
$$
where $\beta_{\max}$ is the boundary depth; see Section
\ref{sec:barcodes} and \cite{Us0,Us}. (Strictly speaking, the Floer
complex in \cite{KS} is restricted to the contractible free homotopy
class by a background assumption, but the inequality still holds
without this restriction.) Therefore, $\gamma\big(\varphi^k\big)>\eps$
whenever $b_\eps\big(\varphi^k\big)$ 
is large enough to guarantee that there is a finite bar of length
greater than $\eps$. For instance, since the Floer persistence module
has exactly $\dim \H_*(M;\F)$ infinite bars, to have the sequence
$\gamma\big(\varphi^k\big)$ bounded away from zero it suffices to
ensure that $b_\eps\big(\varphi^k\big)>\dim\H_*(M;\F)$ for some
$\eps>0$ and all large $k$.

This is the case, for instance, when $\varphi$ has a locally maximal
hyperbolic set $K$ such that for each large $k$ the set $K$ contains
more than $\dim\H_*(M;\F)$ periodic points. For instance, $K$ can be a
horseshoe in the sense of \cite[Sect.\ 6.5]{KH} or, more generally, a
hyperbolic locally maximal invariant subset such that $p(k)\to\infty$
in the notation of Section \ref{sec:pf-B} or just a collection of more
than $\dim\H_*(M;\F)$ hyperbolic fixed points. Alternatively, it is
enough to require that $\dim\H_{\mathit{odd}}(M;\F)=0$ and $K$
contains a periodic orbit of index $m(k)$ such that $m(k)-n$ is odd.
(For instance, this is so when $K$ is a positive hyperbolic fixed
point and $M=S^2$. To prove Proposition \ref{prop:gamma} below in this
case, it is useful to (re)introduce the $\Z_2$-grading on
$\CF(\varphi)$.)  Summarizing these observations we obtain the
following.

\begin{Proposition}
  \label{prop:gamma}
  Assume that $\varphi$ has a locally maximal hyperbolic set $K$ such
  that for every sufficiently large $k$, one of the following two
  conditions is met:

\begin{itemize}

\item[\rm{(i)}] $K$ contains more than $\dim \H_*(M;\F)$ $k$-periodic
  points; or

\item[\rm{(ii)}] $\dim \H_{\mathit{odd}}(M;\F)=0$ and $K$ contains a
  $k$-periodic orbit of index $m(k)$ such that $m(k)-n$ is odd.
  
\end{itemize}
Then the sequences $\gamma\big(\varphi^k\big)$ and
$\big\|\varphi^k\big\|_{\hn}$ are bounded away from zero.
\end{Proposition}
This result is new, although not entirely unexpected to the
authors. Of course, the assertion of Theorem \ref{thm:B} is much
stronger than this proposition; for it guarantees exponential growth
of $b_\eps\big(\varphi^k\big)$ when $\htop(K)>0$. However, the
proposition is more general. We also note that
$\gamma\big(\varphi^k\big)$ is not bounded away from zero
unconditionally when $\varphi^k\neq \id$ for all $k\in\N$. For
instance, $\gamma\big(\varphi^k\big)$ can be arbitrarily small for
pseudo-rotations of $\CP^n$; see \cite{GG:PR}. We revisit connections
between the $\gamma$-norm and dynamics and refine Proposition
\ref{prop:gamma} in \cite{CGG:growth,CGG:generic}.  Overall, little
seems to be known about the behavior of the sequence
$\gamma\big(\varphi^k\big)$.
 
\subsection{Proof of Theorem \ref{thm:C}}
\label{sec:pf-C}
It suffices to show that
\begin{equation}
  \label{eq:hbar-htop}
  \hbr(\varphi)\geq \htop(\varphi)
\end{equation}
whenever $\dim M=2$; for $\hbar(\varphi)\leq \htop(\varphi)$ by
Theorem \ref{thm:A}.

When $M$ is a closed surface,
\begin{equation}
  \label{eq:sup}
  \htop(\varphi)=\sup\big\{\htop(\varphi|_K)\mid
  K \textrm{ is hyperbolic}\big\},
\end{equation}
as a consequence of the results in \cite{Ka80}, and
\eqref{eq:hbar-htop} follows from Theorem \ref{thm:B}. \qed

\begin{Remark}
  Note that in \eqref{eq:sup} it is enough to assume that $\varphi$ is
  only $C^{1+\alpha}$-smooth. Furthermore, we can also require $K$ to
  be locally maximal already by the results from \cite{Ka80}. To see
  this, first recall from \cite[Sect.\ 15.4]{BP} that in dimension two
\begin{equation}
\label{eq:sup2}
\htop(\varphi)=\sup\big\{\htop(\varphi|_K)\mid K \textrm{ is a hyperbolic
  horseshoe}\big\}.
\end{equation}
(This is a consequence of \cite[Thm.\ S.5.9]{KH}, based on
\cite{Ka80}, and two standard results: the Variational Principle,
\cite[Thm.\ 4.5.3]{KH}, and the Ergodic Decomposition Theorem,
\cite[Thm.\ 4.1.12]{KH}, which allows one to restrict the supremum in
the Variational Principle to ergodic measures only.)

Thus, to obtain \eqref{eq:sup} with $K$ locally maximal from
\eqref{eq:sup2}, we just need to make sure that in this context
hyperbolic horseshoes are locally maximal. The definition used in
these results is that a horseshoe is a closed invariant set $K$ such
that $\varphi|_K$ is conjugate to a subshift of finite type (aka, a
topological Markov chain). To be more precise, there is a
decomposition $K=K_0\cup\ldots \cup K_{r-1}$ such that
$\varphi(K_i)=K_{i+1}$, $i\in\Z_r$, and $\varphi^r|_{K_i}$ is
conjugate to a full shift in $s$-symbols; see \cite[Sect.\
S.5.d]{KH}. (In addition, here $K$ is hyperbolic.) Then $K$ has a
local product structure: for any two nearby points the intersection of
the local stable manifold through one of them with the local unstable
manifold through the other is transverse and comprises exactly one
point. This is equivalent to local maximality, see, e.g., \cite[Sect.\
5]{AY}.
\end{Remark}

\end{document}